\documentclass[11pt]{amsart}
\usepackage[babel]{csquotes}
\usepackage{enumitem}
\usepackage{amsmath,amsthm,amssymb,mathrsfs,amsfonts,verbatim,enumitem,color,leftidx,mathtools}
\usepackage{mathabx}
\usepackage{mathdots}
\usepackage{etoolbox} 
\usepackage{bbm}
\usepackage[all,tips]{xy}
\usepackage{graphicx,ifpdf}
\usepackage{stmaryrd}
\usepackage{amssymb}
\usepackage{dsfont}

\ifpdf
\fi
\usepackage[colorlinks]{hyperref}
\hypersetup{
linkcolor=blue,          
citecolor=green,        
}

\usepackage{tikz}

\usepackage{todonotes}
\usepackage{setspace}
\newcounter{mycomment}

\newtheorem{thm}{Theorem}[section]
\newtheorem{lem}[thm]{Lemma}
\newtheorem{cor}[thm]{Corollary}
\newtheorem{prop}[thm]{Proposition}

\theoremstyle{definition}
\newtheorem{defi}[thm]{Definition}

\newtheorem{remark}[thm]{Remark}

\theoremstyle{remark}
\newtheorem{rem}[thm]{Remark}

\numberwithin{equation}{section}

\definecolor{esperance}{rgb}{0.0,0.5,0.0}

\definecolor{RED}{named}{red}

\newcommand{\ba}{\mathbf{a}}
\newcommand{\bb}{\mathbf{b}}
\newcommand{\bc}{\mathbf{c}}
\newcommand{\bd}{\mathbf{d}}
\newcommand{\be}{\mathbf{e}}
\newcommand{\bg}{\mathbf{g}}

\newcommand{\bk}{\mathbf{k}}

\newcommand{\bp}{\mathbf{p}}

\newcommand{\bs}{\mathbf{s}}

\newcommand{\bu}{\mathbf{u}}
\newcommand{\bv}{\mathbf{v}}
\newcommand{\bw}{\mathbf{w}}
\newcommand{\bx}{\mathbf{x}}

\newcommand{\bNN}{\bZ_{\geq 1}}

\newcommand{\MIv}[1]{M_{#1}(\bs,\bv)}

\newcommand{\Mllv}{M_{\llbracket 1, d \rrbracket}(\bs, \bv)}

\newcommand{\Xpt}[1]{X_{#1}(\bs)}

\newcommand{\Xball}[1]{X_{#1}(\bs)}

\newcommand{\qnu}{q_{\nu}}

\newcommand{\cOnu}{\cO_{\nu}}

\newcommand{\pinu}{\pi_{\nu}}

\DeclareMathOperator{\Card}{Card}

\newcommand{\al}{\alpha}

\newcommand{\Ga}{\Gamma}
\newcommand{\del}{\delta}

\newcommand{\lam}{\lambda}
\newcommand{\Lam}{\Lambda}
\newcommand{\eps}{\epsilon}

\newcommand{\om}{\omega}

\newcommand{\vphi}{\varphi}

\newcommand{\cO}{\mathcal{O}}

\newcommand{\cQ}{\mathcal{Q}}

\newcommand{\bE}{\mathbb{E}}

\newcommand{\bR}{\mathbb{R}}
\newcommand{\bZ}{\mathbb{Z}}

\newcommand{\bF}{\mathbb{F}}

\newcommand{\SL}{\operatorname{SL}}

\newcommand{\GL}{\operatorname{GL}}

\newcommand\wt[1]{\widetilde{#1}}

\newcommand\mb[1]{\mathbf{#1}}

\newcommand\tb[1]{\textbf{#1}}

\newcommand{\onto}{\xymatrix{\ar@{>>}[r]&}}

\newcommand{\eq}[1]
{
\begin{equation*}
{#1}
\end{equation*}
}
\newcommand{\eqlabel}[2]
{
\begin{equation}
{#2}\label{#1}
\end{equation}
}

\makeatletter
\newcommand*{\rom}[1]{\expandafter\@slowromancap\romannumeral #1@}
\makeatother

\begin{document}

\title{Singular systems of linear forms over global function fields}

\date{}

\author{Gukyeong Bang}
\address{Gukyeong ~Bang. Department of Mathematical Sciences, Seoul National University}
\email{gukyeong.bang@snu.ac.kr}

\author{Taehyeong Kim}
\address{Taehyeong ~Kim. The Einstein Institute of Mathematics, Edmond J. Safra Campus, Givat Ram\\
The Hebrew University of Jerusalem, Jerusalem, 91904, Israel}
\email{taehyeong.kim@mail.huji.ac.il}

\author{Seonhee Lim}
\address{Seonhee ~Lim. Department of Mathematical Sciences and Research Institute of Mathematics, Seoul National University}
\email{slim@snu.ac.kr}
\thanks{}



\def\thefootnote{}
\def\thefootnote{\arabic{footnote}}
\setcounter{footnote}{0}

\begin{abstract}
In this paper, we consider singular systems of linear forms over global function fields of class number one and give an upper bound for the Hausdorff dimension of the set of singular systems of linear forms by constructing an appropriate Margulis height function on the space of lattices over global function fields.
\end{abstract}

\maketitle
\section{Introduction}
An $m \times n$ matrix $\mb{A} \in M_{m,n}(\bR)$ is called \textit{singular} if for any $\eps>0$ there exists $T_0 \geq 1$ such that for all $T\geq T_0$, there exist $\mb{p}\in\bZ^m$ and $\mb{q}\in\bZ^n$ such that 
\[
\|\mb{A} \mb{q}+\mb{p}\| \leq \eps T^{-n/m}\quad\text{and}\quad 0<\|\mb{q}\|\leq T,
\] where $\|\cdot\|$ denotes the sup norm. In 1926, Khintchine \cite{Khi26} introduced the notion of singularity and showed that the set of such matrices has Lebesgue measure zero. On the other hand, the computation of the Hausdorff dimension of the set of singular matrices has been a challenge until the breakthrough by Das, Fishman, Simmons, and Urba\'{n}ski \cite{DFSU}. Before this breakthrough, it was proved in \cite{C11} that the Hausdorff dimension of the set of singular pairs (i.e., $m=2$ and $n=1$) is $4/3$. This result was extended to the $m$-dimensional singular vectors (i.e., $n=1$) in \cite{CC16}, where the authors proved that the singular set has the Hausdorff dimension $m^2/(m+1)$. For general singular matrices, it was shown in \cite{KKLM} that the Hausdorff dimension of $m\times n$ singular matrices is at most $mn-\frac{mn}{m+n}$, and finally, it was shown in \cite{DFSU} that the upper bound is sharp.

In contrast to the intensive research mentioned above, there have been fewer results on singular sets over function fields (see, for example, \cite{DX}). In this article, we pursue the method in \cite{KKLM} in the function field setting and prove that the Hausdorff dimension of $m\times n$ singular matrices over function fields is at most $mn-\frac{mn}{m+n}$.

Let us first introduce the function field setting we will consider. Let $\mathbf K$ be any global function field over a finite field $\bF_q$ of $q$ elements for a prime power $q$, that is, the function field of a geometrically connected smooth projective curve $\mb{C}$ over $\bF_q$. We fix a discrete valuation $\nu$ on $\mathbf K$ and denote by $\mathbf K_\nu$ and $\cO_\nu$ the completion of $\mathbf K$ with respect to $\nu$ and its valuation ring, respectively. Fix a uniformizer $\pi_\nu\in \mathbf K$ such that $\nu(\pi_\nu)=1$ and denote $\mathbf k_\nu = \cO_\nu /\pi_\nu \cO_\nu$ and $q_\nu = \Card \mathbf k_\nu$. The absolute value $|\cdot|$ associated with $\nu$ is defined by $|x|=q_\nu^{-\nu(x)}$ for any $x\in \mathbf K_\nu$. For any positive integer $s$, we denote by $\|\cdot\|$ the sup norm on $\mathbf K_\nu^s$ and denote by $\dim_H$ the Hausdorff dimension of subsets of $\mathbf K_\nu^s$ with respect to the sup norm.

Let $R_\nu$ be the affine algebra of the affine curve $\mb{C} \smallsetminus \{\nu\}$. Then $R_\nu$ is discrete in $\mathbf K_\nu$ and plays a similar role to the ring of integers $\bZ$ in $\bR$. We remark that $R_\nu$ is a Dedekind domain in general, but we will assume that $R_\nu$ is a principal ideal domain, that is, the ideal class group of $R_\nu$ is trivial (see \cite[II.2]{Ser} and \cite[Chapter 1]{Nar}). The assumption is only needed in order to use the submodularity of the covolume function established in \cite{Bang}.  

We say that a matrix $\mb{A} \in M_{m,n}(\mathbf K_\nu)$ is \textit{singular} if for any $\eps>0$ there exists $T_0\geq 1$ such that for all $T\geq T_0$, there exist
$\mb{p}\in R_\nu^m$ and $\mb{q}\in R_\nu^n$ satisfying   
\eq{
\|\mb{A}\mb{q}+\mb{p}\|\leq \eps T^{-n/m} \quad\text{and}\quad 0<\|\mb{q}\| \leq T. 
}
Denote by $\mathrm{Sing}(m,n)$ the set of all singular matrices in $ M_{m,n}(\mathbf K_\nu)$. 
The main result of this article is the following:
\begin{thm}\label{Thm_sing} 
Assume that $R_\nu$ is a principal ideal domain.
For $m,n \in \bNN$,
\[\dim_H \mathrm{Sing}(m,n) \leq mn - \frac{mn}{m+n}.\]
\end{thm}
In fact, we will prove a stronger statement for the dimension of the set of lattices with trajectories ``escaping on average". See Proposition \ref{Prop_dim} and the definition just before the proposition.

In order to estimate the upper bound of the Hausdorff dimension of the singular set, we mainly use the connection (Subsection \ref{subsec2.2}) between homogeneous dynamics and Diophantine approximation, so-called \textit{Dani's correspondence}, and the construction (Corollary \ref{CorMarFtn}) of an appropriate height function satisfying a certain contraction property, so-called a \textit{Margulis function} which originated in \cite{EMM}.

We basically follow the approach of \cite{KKLM} to construct an appropriate Margulis function. The first step is to obtain a certain contraction property (Proposition \ref{LinAlgEst}) at the linear algebra level, and then move to a contraction property (Corollary \ref{CorAlEst}) in homogeneous space using submodularity.

Gaussian random variables were used in \cite{KKLM} to make calculations simpler at the linear algebra level. In contrast, we will use a restriction of the Haar measure instead. We will prove that some analogous properties of the Gaussian measure still hold for the Haar measure. This is the main step (Theorem \ref{ThmWedgeUpperBound}) to obtain the linear algebra level contraction property.

The article is organized as follows. In Section \ref{sec2}, we introduce terms and definitions related to global function fields, homogeneous dynamics to interpret the singularity in terms of dynamics, and exterior algebra to develop the linear algebra level contraction property.
Throughout Sections \ref{Sec3} and \ref{Sec4}, we will prove the contraction property (Proposition \ref{LinAlgEst}) at the linear algebra level. In particular, Section \ref{Sec3} is devoted to key linear algebra estimates (Theorem \ref{ThmWedgeUpperBound}), and in Section \ref{Sec4}, we prove the linear algebra level contraction property (Proposition \ref{LinAlgEst}) by establishing the Hodge duality over function fields.
Section \ref{sec5} is devoted to the construction of Margulis functions under the assumption that $R_\nu$ is a principal ideal domain.
Finally, in Section \ref{sec6}, we estimate the upper bound of the Hausdorff dimension of the singular set using the Margulis function.

\vspace{5mm}
\tb{Acknowledgments}. We thank Frédéric Paulin for insightful discussions and for sharing his results, and the referee for a careful reading of the manuscript and for valuable suggestions that improved the exposition, in particular for pointing out a more precise and simpler proof of Claim 2 in Proposition \ref{EstProp}.
This project was partly conducted during T.K.'s visit to Seoul National University.
S.L. and G.B. are supported by National Research Foundation
of Korea, under Project number NRF-2020R1A2C1A01011543, S.L. by RS-2023-00301976, RS-2025-02293115, and
RS-2025-00515082. S.L. is an associate member of Korea Institute for Advanced Study.

\section{Preliminaries}\label{sec2}

\subsection{Global function fields}\label{subsec2.1}
As in the introduction, we consider a global function field $\mathbf K$ over a finite field $\bF_q$ of $q$ elements for a prime power $q$.
We denote by $g$ the genus of $\mb{C}$.
Recall that we consider a discrete valuation $\nu$ on $\mathbf{K}$, the completion $\mathbf{K}_\nu$ of $\mb{K}$ for $\nu$, its valuation ring $\cO_\nu$, a uniformizer $\pi_\nu\in\mathbf{K}$, the residual field $\mathbf k_\nu = \cO_\nu / \pi_\nu \cO_\nu$, its cardinality $q_\nu = \Card \mathbf k_\nu$, the absolute value $|\cdot|$ associated with $\nu$, and the affine algebra $R_\nu$ of the affine curve $\mb{C} \smallsetminus \{\nu\}$. 

Then the completion $\mathbf{K}_\nu$ of $\mb{K}$ for $\nu$ is the field $\mathbf k_\nu ((\pi_\nu))$ of Laurent series $x=\sum_{i\in\bZ}x_i(\pi_\nu)^i$ in the variable $\pi_\nu$ over $\mathbf{k}_\nu$, where $x_i\in\mathbf{k}_\nu$ is zero for $i\in\bZ$ small enough.
It follows that 
\[
|x|=q_\nu^{-\sup\{j\in\bZ:\ \forall i<j,\ x_i =0\}},
\]
and $\cO_\nu = \mathbf k_\nu [[\pi_\nu]]$ is the local ring of power series $x=\sum_{i\in\bZ_{\geq 0}}x_i(\pi_\nu)^i$ in the variable $\pi_\nu$ over $\mathbf{k}_\nu$.

For any positive integer $s$, we denote by $\mu_{s}$ the Haar measure of $\mathbf K_\nu ^{s}$ normalized so that $\mu_{s}(\cO_\nu^{s}) = 1$ and by $\overline{\mu_{s}} = \mu_{s}|_{\cO_\nu^{s}}$ the restriction of $\mu_{s}$ on $\cO_\nu^{s}$.
A \textit{lattice} $\Lambda$ in $\mathbf{K}_\nu^s$ is a discrete free $R_\nu$-submodule in $\mathbf{K}_\nu^s$ that generates $\mathbf{K}_\nu^s$ as a $\mathbf{K}_\nu$-vector space. 
The \textit{covolume} of $\Lambda$, denoted by $\mathrm{cov}(\Lambda)$, is the Haar measure $\mu_{s}$ of a fundamental domain of $\Lambda$ in $\mathbf{K}_\nu^s$.
Note that $R_\nu^s$ is a lattice in $\mathbf{K}_\nu^s$ and by \cite[Lemma 14.4]{BPP} we have 
\eqlabel{Eq_cov_R}{
\mathrm{cov}(R_\nu^s)=q^{(g-1)s}.
}
We say that a lattice $\Lam$ in $\mb{K}_\nu^s$ is \textit{unimodular} if $\mathrm{cov}(\Lambda)=\mathrm{cov}(R_\nu^s)$.

\subsection{Homogeneous dynamics}\label{subsec2.2}
Fix $m,n \in \bNN$. Let 
\[
G=\SL_{m+n} (\mathbf K_\nu), \quad \Ga=\SL_{m+n}(R_\nu),\text{ and}\quad X= G/\Ga.
\] 
The homogeneous space $X$ can be identified with the $G$-orbit of $R_\nu^{m+n}$ in the space of unimodular lattices in $\mathbf K_\nu^{m+n}$. Note that the homogeneous space $X$ may not be the full space of unimodular lattices in $\mathbf K_\nu^{m+n}$ as one can multiply the standard lattice $R_\nu^{m+n}$ by a constant which is not in $R_\nu$ but has absolute value 1.
In order to relate Diophantine approximation and homogeneous dynamics, we consider the following diagonal flow: 
\[
g_t = \left(\begin{matrix} \pi_\nu^{-nt} \mathbf{I}_m & \\ & \pi_\nu^{mt} \mathbf{I}_n \end{matrix} \right)\quad \text{for }t\in \bZ.
\] Then $g_t$ acts on $X$ by left multiplication and the unstable horospherical subgroup for $g_1$ in $G$ is given by
\[
U=\left\{u_{\mathbf s}= \left(\begin{matrix} \mathbf{I}_m & \mathbf s \\ & \mathbf{I}_n \end{matrix} \right) : \mathbf s\in M_{m,n}(\mathbf K_\nu)\right\}.
\]  

\begin{lem}[Dani's correspondence]\label{Lem_Dani}
    A matrix $\mathbf{s}\in M_{m,n}(\mathbf{K}_\nu)$ is singular if and only if the trajectory $\{g_t u_{\mathbf{s}}\Ga:t\in\bNN\}$ is divergent in $X$, that is,  leaves every compact subset of $X$.
\end{lem}
\begin{proof}
 The same proof from \cite[\S 2]{D85} works verbatim in our setting with Mahler compactness criterion over global function fields as in \cite[Theorem 1.1]{KST}. 
 \end{proof}

\subsection{Exterior algebra}
{Given $i,d \in \bZ_{\geq 0} $ with $i \leq d$, we denote $\llbracket i, d\rrbracket = \{i, i+1, \dots, d\}$.
For any finite subset $S$ of $\bNN$ satisfying $|S| \geq i$, denote the set of subsets of $i$ elements in $S$ by
\eq{
	\wp_{i}(S) = \{S' \subseteq S : |S'| = i\}
} and denote $\wp_{i}^{d} = \wp_{i}(\llbracket 1, d\rrbracket)$.

Denote $\bv_I = \bv_{j_1} \wedge \cdots \wedge \bv_{j_i}$ for $I = \{j_1, \dots, j_i\} \in \wp^{d}_{i}$ with $j_1 < \cdots < j_i$ and $\bv_{j_1}, \dots, \bv_{j_i} \in \mathbf K_\nu^d$. An element $\bv \in \bigwedge^i {\mathbf K_\nu^d}$ is called \textit{decomposable} if $\bv = \bv_{\llbracket 1, i \rrbracket }$ for some $\bv_1, \dots, \bv_i \in \mathbf K_\nu^d$. Arbitrary elements of $\bigwedge^i {\mathbf K_\nu^d}$ are linear combinations of decomposable elements.
Let $\{\be_1, \dots, \be_d\}$ be the standard basis of $\mathbf K_\nu^d$. Then $\{\be_I \}_{I \in \wp^{d}_{i}}$ is the standard basis of $\bigwedge^i {\mathbf K_\nu^d}$. Define the norm on $\bigwedge^{i} \mathbf K_\nu^d$ by 
\eq{
	\| \bv\| = \max_{I \in \wp^{d}_{i}}|v_I|.
} for $\bv = \sum_{I \in \wp^{d}_{i}}v_I \be_I \in \bigwedge^{i} \mathbf K_\nu^d$.

Note that any element $g\in \SL_d(\mathbf{K}_\nu)$ acts linearly on $\mathbf K_\nu^{d}$. The action naturally extends to 
$\bigwedge^{i} \mathbf K_\nu^{d}$ by 
$$g.(\bv_1 \wedge \dots \wedge \bv_i) = (g\bv_1) \wedge \dots \wedge (g\bv_i).$$
We denote the orthogonal group for a non-Archimedean field by 
\eqlabel{def_ortho}{
	\GL_d(\cO_\nu) =  \{ A \in M_{d,d}(\cO_\nu) : |\det{A}| = 1  \},
}
and the special orthogonal group by $\SL_d(\cO_\nu) = \GL_d(\cO_\nu)  \cap \SL_d(\mathbf K_\nu)$. 
Note that the non-Archimedean norm on $\mathbf K_\nu^d$ is invariant under $\GL_d(\cO_\nu)$ (see \cite{Wei}), and thus,
the restriction measure $\overline{\mu_{d}}$ is $\GL_d(\cO_\nu)$-invariant.

\section{Key linear algebra estimates}\label{Sec3}
Throughout Sections \ref{Sec3} and \ref{Sec4}, we will prove the following proposition, which is a global function field analog of \cite[Proposition 3.1 and Lemma 3.5]{KKLM}.
We remark that all statements in Sections \ref{Sec3} and \ref{Sec4} actually hold for all non-Archimedean local fields, including $p$-adic fields and completions of global function fields, without any condition on the class number.

\begin{prop}\label{LinAlgEst}
For any $i \in \llbracket 1, m+n-1 \rrbracket$, let 
$$\beta_i := \begin{cases} \frac{m}{i} &\quad \text{if } i\leq m, \\ \frac{n}{m+n-i} &\quad \text{if } i>m.\end{cases}$$
Then there exists $c>0$ depending only on $m,n$ such that 
\eq{
\int_{M_{m,n}(\cOnu)} \|g_t u_\bs \bv\|^{-\beta_i} d\mu_{mn}(\bs) \leq ct \qnu^{-mnt}\|\bv\|^{-\beta_i}}
for any $t \in \bNN$ and any decomposable $\bv \in \bigwedge^{i} \mathbf K_\nu^{m+n}$.
\end{prop}

The goal of this section is to prove Proposition \ref{OrtEst}, which is the first step of Proposition \ref{LinAlgEst}.
Note that Proposition \ref{OrtEst} is an analog of \cite[Proposition 3.1]{KKLM}, but restricted to wedge vectors of dimension at most $m$, due to the fact that the Hodge dual operator is no longer equivariant under the orthogonal group $\GL_d(\cOnu)$.
For example, for $A= \left(\begin{matrix} 1 & 1 \\ 0 & 1 \end{matrix} \right) \in \GL_2(\cOnu)$, $*(A \be_1 ) \neq A(*\be_1)$ (see Lemma \ref{lem_dual} for detail). 

Another remark is that while the Gaussian measure is used in \cite{KKLM} to control an integral over $\mathrm{O}_d(\bR)$, we will use the restricted Haar measure $\overline{\mu_d}$ as an analog.
In \cite[Theorem 4.2]{Eva01}, it is shown that every $\GL_d(\cOnu)$-invariant probability measure on $\mathbf K_\nu^d$ whose coordinate projections are independent is a normalized restricted measure of the Haar measure on a $\cOnu$-module $\pinu^r \cOnu^d$ for some $r \in \bZ$.
Considering Maxwell's characterization of the Gaussian distribution in the real case, $\overline{\mu_d}$ is the right analog of (real) Gaussian measure.

Throughout this section, let $i,d \in \bNN$ with $i \leq d$. 
Given a matrix $\bs$, we denote the $j$-th column and $j$-th row vectors by $\bs_j$ and $\bs^j$, respectively.
For $\bs \in M_{d,i}(\cOnu)$, let $c_I(\bs)$ be the coefficient corresponding to $\be_I$ of $\bs_1 \wedge \dots \wedge \bs_i \in \bigwedge^i \mathbf K_\nu^d$, i.e.,
\eq{
	\bs_1 \wedge \dots \wedge \bs_i = \sum_{I \in \wp^{d}_{i}} c_I(\bs) \be_I.
} 
In order to prove Proposition \ref{OrtEst}, we need the following measure estimates:
\begin{thm}\label{ThmWedgeUpperBound}
Given $\ell, \ell ' \in \mathbb Z_{\geq 0}$ with $\ell \leq \ell'$, let $E_\ell (d,i)$ be the set of $i$-tuples of vectors of $\cOnu^d$ whose wedge product has norm bounded by $q^{-\ell}$, and $F_{\ell,\ell'}$ be its subset with the absolute value of the ``first'' coordinate $c_{  \llbracket 1, i\rrbracket}(\bs)$ bounded by $q^{-\ell'}$:
\begin{align} \label{ThmWedgeUpperBound_1}
	&E_{\ell}(d, i) = \{\bs = (\bs_1, \dots, \bs_i) \in M_{d,i}(\cOnu) :  \|  \bs_1 \wedge \dots \wedge \bs_i \|   \leq  \qnu^{-\ell}
	\}; \\
    &F_{\ell, \ell'}(d, i) = E_{\ell}(d, i) \cap \{ \bs \in M_{d,i}(\cOnu) :  |  c_{\llbracket 1, i\rrbracket}(\bs) |   \leq  \qnu^{-\ell'} \}. \nonumber    
\end{align} 
For each $i \in \llbracket 1, d \rrbracket$, there exists a constant $C_1(d,i) > 0$ depending on $d$ and $i$ such that
\eq{
	\mu_{di}(F_{\ell, \ell'}(d, i)) \leq C_1 \binom{\ell' - \ell + \binom{d}{i} }{\ell'-\ell+1} \qnu^{-\ell(d - i) -\ell'}.
} In particular, $\mu_{di}(E_{\ell}(d, i)) \leq C_2 \qnu^{-\ell(d - i + 1)}$ for a constant $C_2 = C_2(d, i) > 0$. Consequently, taking $C_1(d) := \max_{1 \leq i \leq d} C_1(d,i)$ and $C_2(d) := \max_{1 \leq i \leq d} C_2(d,i)$ yields constants depending only on $d$.
\end{thm}
Given a $\binom{d}{i}$-tuple $(n_I)_{I \in \wp_{i}^d }$ of non-negative integers, consider even more refined subsets
\eq{
	D_{I}^{d}(\ell) = \{ \bs \in M_{d, i}(\cOnu) : | c_I(\bs) |  =  \qnu^{-\ell} \} \text{ and } D^d_i((n_I)_{I}) = \bigcap_{I \in \wp_{i}^d }D_{I}^{d}(n_I).
}
In other words, the set $D^d_i((n_I)_{I})$ is the subset of $\bigwedge^i \mathbf \cOnu^d$ consisting of elements whose coefficient $c_I$ has absolute value exactly $q^{-{n_I}}$, for all $I$.

Since we use the sup norm for $\bigwedge^i \mathbf K_\nu^d$, the set $F_{\ell, \ell'}(d, i)$ is partitioned into
\eq{
  F_{\ell, \ell'}(d, i) = \coprod_{
  (n_I)_{I \in \wp_{i}^d} \, \in \,  \bNN^{\wp_{i}^d} \, : \, n_{\llbracket 1, i \rrbracket} \geq \ell'; \,
  \forall I \in \wp_{i}^d, \, n_{I} \geq \ell
  } 
  D^d_i((n_I)_{I}).
} We will simultaneously prove the following estimate of $\mu_{di}(D^d_i((n_I)_{I}))$. 
We define $\exp_{\qnu}(x) = \qnu^x$ for every $x \in \bR$.
\begin{lem}\label{LemWedgeUpperBound}
For $d, i\in \bNN$ with $i \leq d$, there exists a constant $C = C(d,i) > 0$ such that for any $\binom{d}{i}$-tuple $(n_I)_{I \in \wp_{i}^d}$ of non-negative integers, there exists $I' \in \wp_{i}^d$ such that we have

\eq{
  \mu_{di}(D^d_i((n_I)_{I})) \ \leq \ C \exp_{\qnu}(-\max_{I \in \wp_{i}^d} {n}_{I} - (d-i) {n}_{I'}).
}
\end{lem}
\begin{remark} The proofs of Lemma \ref{LemWedgeUpperBound} and Theorem \ref{ThmWedgeUpperBound} are carried out simultaneously: while Lemma \ref{LemWedgeUpperBound} for $(d,i)$ implies Theorem \ref{ThmWedgeUpperBound} for $(d,i)$, the proof of Lemma \ref{LemWedgeUpperBound} for $(d,i)$ uses \eqref{ThmWedgeUpperBound_1} from Theorem \ref{ThmWedgeUpperBound} for $(d, i-1)$.
\end{remark}

\begin{proof}[Proofs of Theorem \ref{ThmWedgeUpperBound} and Lemma \ref{LemWedgeUpperBound}]
We proceed by induction on $(d, i)$ with $1 \leq i \leq d$, using the following four steps.

\textbf{Step 0} : Lemma \ref{LemWedgeUpperBound} holds for $(d, 1)$ with $d \geq 1$.

\textbf{Step 1} : If Lemma \ref{LemWedgeUpperBound} holds for $(d, i)$ with $2 \leq i \leq d$, then Lemma \ref{LemWedgeUpperBound} holds for $(d+1, i)$.

\textbf{Step 2} : If Lemma \ref{LemWedgeUpperBound} holds for $(d, i)$ with $1 \leq i \leq d$, then Theorem \ref{ThmWedgeUpperBound} holds for $(d, i)$.

\textbf{Step 3} : If Lemma \ref{LemWedgeUpperBound} and Theorem \ref{ThmWedgeUpperBound} hold for $(d+1, d)$ with $d \geq 1$, then Lemma \ref{LemWedgeUpperBound} holds for $(d+1, d+1)$.

By \textbf{Step 0}, Lemma \ref{LemWedgeUpperBound} holds for $(d, 1)$ with $d \geq 1$. In particular, Lemma \ref{LemWedgeUpperBound} holds for $(2,1)$, which implies Theorem \ref{ThmWedgeUpperBound} for $(2,1)$ by \textbf{Step 2}. Then, \textbf{Step 3} ensures that Lemma \ref{LemWedgeUpperBound} holds for $(2,2)$. For the remaining cases $2 \leq j \leq d$, assume by induction that Lemma \ref{LemWedgeUpperBound} holds for $(j, j)$ with $j \geq 2$. Then Lemma \ref{LemWedgeUpperBound} holds for $(d', j)$ with $2 \leq j \leq d'$ by iterating \textbf{Step 1}, and in particular for $(j+1, j)$ if $j < d$. Hence, Theorem \ref{ThmWedgeUpperBound} holds for $(j+1, j)$ by \textbf{Step 2}. Therefore, Lemma \ref{LemWedgeUpperBound} holds for $(j+1, j+1)$ by \textbf{Step 3}. The two results follow by induction. \\[0.5em]
\noindent \textbf{Proof of Step 0.} Denote $n_j:=n_{\{j\}}.$ Since
\eq{
	D^d_1((n_I)_{I}) = \{\mathbf{s}=(s_{1}, \dots, s_{d}) \in \cOnu^{d} : |s_{j}| = \qnu^{-n_{j}}, \;\;  j \in \llbracket 1, d \rrbracket \ \}, 
} the Haar measure of this set is
\eq{
	\begin{split}
		\mu_{d}(D^d_1((n_I)_{I})) 
		& = \prod_{j = 1}^{d} \mu_{1} (\{s_{j} \in \cOnu : |s_{j}| = \qnu^{-n_{j}}\}) = \prod_{j = 1}^{d} (\frac{1}{\qnu^{ n_{j}}}  - \frac{1}{\qnu^{ n_{j} +1 }}  ) \\
  & \leq \exp_{\qnu}(- \sum_{j=1}^{d} n_{j}) \leq \exp_{\qnu}(- \max\limits_{1\leq j\leq d} n_{j} - (d-1)\min\limits_{1 \leq j \leq d} n_{j} ),
 \end{split}
} which implies Lemma \ref{LemWedgeUpperBound} for $(d, 1)$ for any $d \in \bNN$.\\[0.5em]
\noindent \textbf{Proof of Step 1.} Suppose that Lemma \ref{LemWedgeUpperBound} holds for $(d, i)$ with $2 \leq i \leq d$. 
We will divide $(d+1)i$ variables into $di$ and $i$ variables in several ways, and apply Fubini's theorem.
We will obtain several upper bounds of partial integral (\textbf{Claim 1}),
and find an appropriate bound among them (\textbf{Claim 2}). 
Finally, we use the induction hypothesis for another partial integral such that we have the required upper bound for the measure of $D^{d+1}_i((n_I)_I)$.

Fix $w \in \llbracket 1, d+1\rrbracket $. We will choose an optimal $w$ later in \textbf{Claim 2}.
Let $\{\be_{kl}\}_{1 \leq k \leq d+1, 1 \leq l \leq i}$, $\{\mathbf{f}_{kl}\}_{1 \leq k \leq d, 1 \leq l \leq i}$, and $\{\be_{l}\}_{1 \leq l \leq i}$ be the standard $\cOnu$-bases of $M_{d+1, i}(\cOnu)$, $M_{d, i}(\cOnu)$, and $\cOnu^i$, respectively. 
Define $\cOnu$-linear maps $\phi = \phi(d,i,w)$ from $M_{d+1, i}(\cOnu)$ to $M_{d, i}(\cOnu)$ suppressing $w$-th row and shifting the following rows by 1, and $\psi = \psi(d,i,w)$ from $M_{d+1, i}(\cOnu)$ to $\cOnu^i$ as follows: let the map $\theta_{w}:\llbracket 1, d\rrbracket \to \llbracket 1, d+1\rrbracket$ be given by $\theta_{w}(k) = k$ if $k < w$ and $\theta_{w}(k) = k+1$ otherwise, and define
\eq{
	\phi(\be_{k l}) =
	\begin{cases}
		\mathbf{f}_{\theta_{w}^{-1}(k) l} &\text{ if } k \ne w \\
		\textbf{0} &\text{ if } k = w,
	\end{cases} \quad \quad
	\psi(\be_{k l}) = \begin{cases}
		\be_{l} &\text{ if } k = w \\
		\textbf{0} &\text{ if } k \ne w .
	\end{cases} 
} 
Note that these maps are well-defined for any integers $d,i\ge1$ and 
$w\in\llbracket1,d+1\rrbracket$, independently of the assumption $2 \le i\le d$ used in this step.
Denote $\theta_{w}(\wp_{i}^{d}) = \{\theta_{w}( \tilde{I} ) : \tilde{I} \in \wp_{i}^{d} \}$, 
i.e., the set of indices $I \in \wp^{d+1}_i $ which do not contain $w$.
Note that the product map $\phi \times \psi$ from $M_{d+1, i}(\cOnu)$ to $M_{d, i}(\cOnu) \times \cOnu^i$ is simply shifting the $w$-th row to the last row and shifting the other rows accordingly. It is clearly bijective, linear and the pushforward measure of $\overline{\mu_{di}} \times \overline{\mu_{i}}$ by $(\phi \times \psi)^{-1}$ is $\overline{\mu_{(d+1)i}}$ by the uniqueness of probability Haar measure on $M_{d+1,i}(\cOnu)$.

Let $D = D^{d+1}_i((n_I)_I)$ and $\overline{D}_w = D^d_i((n_{\theta_{w}( \tilde{I} )})_{\tilde{I} \in \wp_{i}^{d}})$. 
By ignoring the conditions for indices $I$ containing $w$, we have
\eq{
	D \subseteq \bigcap_{I \in \theta_{w}(\wp_{i}^{d}) } D_{I}^{d+1}(n_I) 
	= \phi^{-1}(\overline{D}_w).
}
Then $(\phi \times \psi)(D) \subseteq \overline{D}_w \times \cOnu^i$.
Using Fubini's theorem, we have
\eqlabel{Eq_Fubini}{
\begin{split}
\mu_{(d+1)i}(D) 
&= \int_{(\phi \times \psi)^{-1}(\overline{D}_w \times \cOnu^i)} \mathds{1}_{D}(\bs) d\mu_{(d+1)i}(\bs) \\
&= \int_{\overline{D}_w} \int_{\cOnu^{i}} \mathds{1}_{D}((\phi \times \psi)^{-1}(\bs_1, \bs_2))  d\mu_i(\bs_2) d\mu_{di}(\bs_1).
\end{split}
}

For a matrix $\bs \in  \overline{D}_w$, 
an index $J= \{k_1, \cdots, k_{i-1}\} \in \wp_{i-1}^{d}$ with $k_1 < \cdots < k_{i-1}$
and $\bv \in \cOnu^{i}$, let  
$\MIv{J}$ be the $i\times i$ matrix whose $j$-th row is the $k_j$-th row vector of $\bs$ for $j\in \llbracket 1, i-1\rrbracket$ and the last row is $\bv.$
Given $H \in \bigcup_{j \geq i-1} \wp_{j}^{d} $, define
\eq{
	\Xpt{H} = \{ \bv \in \cOnu^{i} : 
| \text{det} \MIv{J} |  = \ \exp_{\qnu}(-n_{ \theta_{w}(J) \sqcup \{w\}}),\ \ \forall J \in \wp_{i-1}(H) \}.	
} 

Observe that for any $\bs_1 \in \overline{D}_w$ and $\bs_2 \in \cOnu^i$, we have
\eqlabel{Eq_X_Rep}{
\mathds{1}_{D}((\phi \times \psi)^{-1}(\bs_1, \bs_2)) \leq \mathds{1}_{X_{\llbracket 1, d \rrbracket}(\bs_1)}(\bs_2)
}
since if $\bs = (\phi \times \psi)^{-1}(\bs_1, \bs_2) \in D$, then for any $J \in \wp_{i-1}^{d}$, 
\eq{
|\text{det}M_J(\bs_1, \bs_2)|=|c_{\theta_{w}(J) \sqcup \{w\}} (\bs)| = \exp_{\qnu}(-n_{ \theta_{w}(J) \sqcup \{w\}}),
}
where the first equality follows from the definitions of $\phi$ and $\psi$: 
indeed, $\bs_1 = \phi(\bs)$ consists of the rows of $\bs$ with indices in $\theta_w(J)$ 
and $\bs_2 = \psi(\bs)$ is the $w$-th row of $\bs$; 
and the second equality follows from the assumption $\bs \in D$ and the definition of the set $D$.

\noindent \textbf{Claim 1.} For $\bs \in \overline{D}_w$ and $\hat{I} \in \wp_{i}^{d},$ we have
$$	\mu_{i}( \Xball{\hat{I}} ) \leq \exp_{\qnu}\Bigl( (i-1)n_{\theta_w(\hat{I})} - \sum\limits_{J \in \wp_{i-1}(\hat{I})} n_{\theta_w(J) \sqcup \{w\}} \Bigr).$$
\begin{proof}[Proof of Claim 1]
Set $\hat{I} = \{ \iota_1, \dots, \iota_i\}$. 
Consider an $(i, i)$-matrix $\bs_{\hat{I}}$ obtained by keeping the $i$-th rows of $\mathbf s$, for $ i \in \hat{I}$. For each $1 \leq k, l \leq i$, let $\rho_{k l}(\bs)$ be the $(k, l)$-entry of the adjugate matrix $\text{adj}(\bs_{\hat{I}})$ of $\bs_{\hat{I}}$.
Note that for any $k \in \llbracket 1, i \rrbracket$ and $\bv = (v_j) \in \cOnu^{i}$,
\eq{
	|\text{det}\MIv{\hat{I} - \{ \iota_k \} }| = |\rho_{k 1}(\bs)v_1 + \cdots + \rho_{k i}(\bs)v_i|.
} 

Set $n_k:=n_{\theta_w(\hat{I} - \{ \iota_k \}) \sqcup \{w\} }$ for the rest of the proof of Claim 1.  By considering $\text{adj}(\bs_{\hat{I}})$ as a linear transformation,
\eq{
	\begin{split}
		\Xball{\hat{I}}
	& = \{ \bv \in \cOnu^{i} : | \text{det}\MIv{\hat{I} - \{ \iota_k \} } |  = \qnu^{-n_k} ,\ k \in \llbracket 1, i \rrbracket \} \\
	& = \{ \bv \in \cOnu^{i} : | \sum_{l =1}^{i} \rho_{k l}(\bs) v_l |  = \qnu^{-n_k} , \ k \in \llbracket 1, i \rrbracket \} \\
	& = \text{adj}(\bs_{\hat{I}})^{-1}(\{({v}_{1}', \dots, {v}_{i}') \in \mathbf K_{\nu}^{i} : |{v}_{k}'| = \qnu^{-n_k}, \ k \in \llbracket 1, i \rrbracket \}) .
	\end{split}
} By the property of the adjugate matrix, we have
\eqlabel{Eq_X_Transform}{
	\begin{split}
	\mu_{i}(\Xball{\hat{I}})	
	&= |\text{det}(\text{adj}(\bs_{\hat{I}})^{-1})|  \, \mu_{i} (\prod_{k=1}^{i} \{{v}_{k}' : |{v}_{k}'| = \qnu^{-n_k} \} )  \\
	&\leq |\text{det}(\text{adj}(\bs_{\hat{I}})^{-1})| \prod_{k=1}^{i} \qnu^{-n_k} \\
	&= |\text{det}(\bs_{\hat{I}})^{-(i-1)}| \exp_{\qnu}( - \sum\limits_{k=1}^{i} n_k).
	\end{split}
}
 Since $\bs \in \overline{D}_w$ and $\hat{I} \in \wp_{i}^{d}$, we have $\bs \in D_{\hat{I}}^{d}(n_{\theta_w(\hat{I})})$,
 and thus $|\text{det}(\bs_{\hat{I}})| = |c_{\hat{I}}(\bs)| = \qnu^{-n_{\theta_w(\hat{I})}}$. Combining with \eqref{Eq_X_Transform}, we have
\eq{
	\mu_{i}( \Xball{\hat{I}} ) \leq \exp_{\qnu}\Bigl( (i-1)n_{\theta_w(\hat{I})} - \sum\limits_{k=1}^{i} n_{\theta_w(\hat{I}- \{ \iota_k \} ) \sqcup \{w\} } \Bigr).
}
\end{proof}
Hence, for any $\hat{I} \in \wp_i^d$, it follows from \eqref{Eq_Fubini}, \eqref{Eq_X_Rep}, and \textbf{Claim 1} that
\eqlabel{Eq_Wedge_Main_Bound}{
\begin{split}
\mu_{(d+1)i}(D) 
&\leq \int_{\overline{D}_w} \int_{\cOnu^{i}} \mathds{1}_{X_{\llbracket 1, d \rrbracket}(\bs_1)}(\bs_2)  d\mu_i(\bs_2) d\mu_{di}(\bs_1) \\
&\leq \int_{\overline{D}_w} \int_{\cOnu^{i}} \mathds{1}_{X_{\hat{I}}(\bs_1)}(\bs_2)  d\mu_i(\bs_2) d\mu_{di}(\bs_1) \\
&\leq \int_{\overline{D}_w} \exp_{\qnu} \Bigl( (i-1)n_{\theta_w(\hat{I})} - \sum\limits_{J \in \wp_{i-1}(\hat{I})} n_{\theta_w(J) \sqcup \{w\}} \Bigr) d\mu_{di}(\bs_1) \\
&= \exp_{\qnu} \Bigl( (i-1)n_{\theta_w(\hat{I})} - \sum\limits_{J \in \wp_{i-1}(\hat{I})} n_{\theta_w(J) \sqcup \{w\}} \Bigr) \mu_{di}(\overline{D}_w).
\end{split}
}

\noindent \textbf{Claim 2.}\; There exist $w_0 \in \llbracket 1, d+1\rrbracket$ and $I_0 \in \theta_{w_0}(\wp_{i}^{d})$ satisfying
\eqlabel{Eq_Max_Finder}{
	(i-1)n_{I_0} - \sum\limits_{J \in \wp_{i-1}(I_0)} n_{J \sqcup \{ w_0 \}} \leq -\max_{I \in \wp_{i}^{d+1}} n_I.
}
\begin{proof}[Proof of Claim 2]
Let $I_{\max}$ be an element in
$\wp_{i}^{d+1}$ such that ${n}_{I_{\max}}= \max\limits_{I \in \wp_{i}^{d+1}} n_I$.
Consider the subset $\cQ$ of $\wp_i^{d+1}$ whose element $I$ differs from $I_{\max}$ by one digit, i.e., $|I_{\max} \smallsetminus I | =1$.
Let $I_0 \in \cQ$ be such that $n_{I_0} = \min\limits_{I \in \cQ} n_I$. Note that by construction, there exist $w_0, w_1 \in \llbracket 1, d+1\rrbracket$, and $J_0 \in \wp_{i-1}(I_{\max})$ satisfying $I_{\max} = J_0 \sqcup \{ w_0\}$, $I_0 = J_0 \sqcup \{ w_1\}$, $w_0 \notin I_0$, and $w_1 \notin I_{\max}$.

We claim that the pair ($w_0$, $I_0$) satisfies \eqref{Eq_Max_Finder}.
Since $ I_0$ does not contain $w_0$, clearly $I_0 \in \theta_{w_0}(\wp_{i}^{d})$.
Let $J \in \wp_{i-1}(I_0) \smallsetminus \{ J_0 \}$. Then we have $w_1 \in J$ and $J \smallsetminus \{ w_1 \} \subset J_0 \subset I_{\max}$. By definition, we have $(J \smallsetminus \{w_1 \}) \sqcup \{ w_0 \} \in \wp_{i-1}{(I_{\max})}$. Hence, $J \cup \{ w_0 \} = ((J \smallsetminus \{ w_1 \}) \sqcup \{ w_0 \} ) \sqcup \{ w_1 \} \in \cQ$, which implies that 
\eq{
	n_{J \cup \{ w_0 \} } \geq \min_{I \in \cQ} n_I = n_{I_0}.
}
Since $J_0 \sqcup \{w_0\} = I_{\max}$, it follows that
\eq{
	 \sum_{J \in \wp_{i-1}(I_0)} n_{J \sqcup \{w_0\}} = n_{J_0 \sqcup \{w_0\}} + \sum_{J \in \wp_{i-1}(I_0) \smallsetminus \{J_0\}} n_{J \sqcup \{w_0\}} \geq  n_{I_{\max}} + (i-1)n_{I_0},
} which shows that ($w_0$, $I_0$) satisfies (\ref{Eq_Max_Finder}).
\end{proof}
Hence, it follows from \eqref{Eq_Wedge_Main_Bound} and \textbf{Claim 2} that
\eqlabel{Eq_Wedge_Main_Bound2}{
	\mu_{(d+1)i}(D) \leq  \exp_{\qnu}{(-\max_{I \in \wp_{i}^{d+1}}n_I)} \mu_{di}(\overline{D}_{w_0}).
}

By the induction hypothesis that Lemma 3.3 holds for $(d, i)$, there exist a constant $C = C(d,i) > 0$ and $\tilde{I}' \in \wp_i^d$ such that
\eq{
\begin{split}
 \mu_{di}(\overline{D}_{w_0}) 
 &\leq C \exp_{\qnu}{(-\max_{\tilde{I} \in \wp_{i}^d } {n}_{\theta_{w_0}(\tilde{I})}  - (d-i) {n}_{\theta_{w_0}(\tilde{I}')})} \\
 &\leq  C \exp_{\qnu}{(- (d-i+1) {n}_{\theta_{w_0}(\tilde{I}')})}.    
\end{split}
} 
Combining with (\ref{Eq_Wedge_Main_Bound2}), we obtain
\eq{
	\mu_{(d+1)i}(D)
	\leq C \exp_{\qnu}{(-\max_{I \in \wp_{i}^{d+1}} n_I  - (d-i+1) {n}_{\theta_{w_0}(\tilde{I}')})}.
} Therefore, by letting $I' := \theta_{w_0}(\tilde{I}') \in \wp_{i}^{d+1}$, Lemma \ref{LemWedgeUpperBound} holds for $(d+1, i)$, which concludes \textbf{Step 1}.\\[0.5em]
\noindent \textbf{Proof of Step 2.} 
Suppose that Lemma \ref{LemWedgeUpperBound} holds for $(d, i)$ with $1 \leq i \leq d$.
Denote $\wp_{i}^d = \{I_1, I_2, \cdots, I_k\}$ with $k = \binom{d}{i}$ and define
\eq{
	T_{\ell, \ell'} = \coprod_{(n_I)_{I \in \wp_{i}^d} \, \in \,  \bNN^{\wp_{i}^d} \, : \,  \ell \leq n_{I_1} \leq \dots \leq n_{I_k}, \ \ell' \leq n_{\llbracket 1, i \rrbracket}} D^d_i((n_I)_{I})
} and for $r \in \llbracket 0, k \rrbracket$, 
\eq{
	S_{\ell, \ell'}(r) = \coprod_{(n_I)_{I \in \wp_{i}^d} \, \in \,  \bNN^{\wp_{i}^d} \, : \,  \ell \leq n_{I_1} \leq \dots \leq n_{I_r} \leq \ell' \leq n_{I_{r+1}} \leq \dots \leq n_{I_k}} D^d_i((n_I)_{I}).
} Since $\llbracket 1, i \rrbracket \in \wp_{i}^{d}$ and $\ell' \leq n_{\llbracket 1, i \rrbracket}$ in $T_{\ell, \ell'}$, we have
\eqlabel{Eq_Wedge_Sym}{
	T_{\ell, \ell'} \subseteq \bigcup_{r = 0}^{k-1} S_{\ell, \ell'}(r).
} By Lemma \ref{LemWedgeUpperBound} for $(d,i)$, there exists a constant $C = C(d,i) > 0$ such that
\eq{
	\mu_{di}(D^d_i((n_I)_{I})) \ \leq \ C \exp_{\qnu}{(-\max_{I \in \wp_{i}^d} {n}_{I} - (d-i) \min_{I \in \wp_{i}^d} {n}_{I})}.
}
Hence, we can estimate the Haar measure of $S_{\ell, \ell'}(r)$ as
\begin{align}
	&\mu_{di}(S_{\ell, \ell'}(r)) 
	= \sum_{{\ell \leq n_{I_1} \leq \dots \leq n_{I_r} \leq \ell' \leq n_{I_{r+1}} \leq \dots \leq n_{I_k}}} \mu_{di}(D^d_i((n_I)_{I})) \nonumber \\
	&\leq \sum_{{\ell \leq n_{I_1} \leq \dots \leq n_{I_r} \leq \ell' \leq n_{I_{r+1}} \leq \dots \leq n_{I_k}}} C \exp_{\qnu}{(-\max_{I \in \wp_{i}^d} {n}_{I} - (d-i) \min_{I \in \wp_{i}^d} {n}_{I}) } \nonumber \\
	&\leq \sum_{n_{I_r} = \ell}^{\ell'} \cdots \sum_{n_{I_1} = \ell}^{n_{I_2}}  \sum_{n_{I_{r+1}} = \ell'}^{\infty} \cdots \sum_{n_{I_{k}} = n_{I_{k-1}}}^{\infty} C \qnu^{-n_{I_k} - (d-i)\ell}\ \ \ \ \ \ \ \ \ \ \ \ \ \ \ \ \ \ \  \nonumber
\end{align}
\begin{align}
	&\leq \sum_{n_{I_r} = \ell}^{\ell'} \cdots \sum_{n_{I_1} = \ell}^{n_{I_2}}  \sum_{n_{I_{r+1}} = \ell'}^{\infty} \cdots \sum_{n_{I_{k-1}} = n_{I_{k-2}}}^{\infty} C (1 - \qnu^{-1})^{-1} \qnu^{-n_{I_{k-1}} - (d-i)\ell} \nonumber \\
	&\ \ \ \ \ \ \ \ \ \ \ \ \ \ \ \ \ \ \ \ \ \ \ \ \ \ \ \ \ \ \ \ \ \ \ \ \ \ \ \ \ \ \ \ \ \vdots \nonumber \\
	&\leq \sum_{n_{I_r} = \ell}^{\ell'} \cdots \sum_{n_{I_1} = \ell}^{n_{I_2}}  C (1 - \qnu^{-1})^{-k+r} \qnu^{-{\ell'} - (d-i)\ell} \nonumber \\
	&\leq C (1 - \qnu^{-1})^{-k} \qnu^{-{\ell'} - (d-i)\ell} \sum_{n_{I_r} = \ell}^{\ell'} \cdots \sum_{n_{I_1} = \ell}^{n_{I_2}} 1. \nonumber \\
	&= C (1 - \qnu^{-1})^{-k} \qnu^{-{\ell'} - (d-i)\ell} \binom{\ell' - \ell + r}{r}, \nonumber
\end{align}
where the last equality holds by setting $m_{I_{j}} = n_{I_{j}} - (\ell - 1)$ and using a simplicial polytopic number
\eq{
	\sum_{n_{I_r} = \ell}^{\ell'} \cdots \sum_{n_{I_1} = \ell}^{n_{I_2}} 1 = \sum_{m_{I_r} = 1}^{\ell'-\ell+1} \cdots \sum_{m_{I_1} = 1}^{m_{I_2}} 1 = \binom{\ell' - \ell + r}{r}.
}
Hence, it follows from \eqref{Eq_Wedge_Sym} and Fermat's identity that
\eq{
	\begin{split}
		\mu_{di}(T_{\ell, \ell'}) &\leq C (1 - \qnu^{-1})^{-k} \qnu^{-\ell' - (d-i)\ell} \sum_{r = 0}^{k-1}  \binom{\ell' - \ell + r}{r} \\
		&\leq C (1 - \qnu^{-1})^{-k} \qnu^{-\ell' - (d-i)\ell} \binom{\ell' - \ell + k}{k - 1},
	\end{split}
} and thus, we obtain by symmetry that
\eqlabel{Eq_Wedge_Sym_Final}{
	\mu_{di}(F_{\ell,\ell'}(d, i)) \leq (k!) \mu_{di}(T_{\ell, \ell'}) \leq C_1  \qnu^{-\ell' - (d-i)\ell} \binom{\ell' - \ell + k}{k - 1},
} where $C_1 = C \max\limits_{1 \leq i \leq d}{k!} (1 - \qnu^{-1})^{-k} $ is a constant depending on $d$ and $i$.

For $E_{\ell}(d,i)$, apply $\ell = \ell'$ to \eqref{Eq_Wedge_Sym_Final}. Then
\eq{
	\mu_{di}(E_{\ell}(d, i)) \leq C_1 k \qnu^{-(d-i + 1)\ell} \leq C_2  \qnu^{-(d-i + 1)\ell},
} where $C_2 = C_1\max\limits_{1 \leq i \leq d}k$ depends on $d$ and $i$. \\[0.5em]
\noindent \textbf{Proof of Step 3.} 
Suppose that Lemma \ref{LemWedgeUpperBound} holds for $(d+1, d)$ and consider the case of $(d+1, d+1)$.
Since $\wp_{d+1}^{d+1} = \{\llbracket 1, d+1 \rrbracket\}$, it is sufficient to show that there exists a constant $C=C(d) > 0$ such that for any $N \geq 0$, we have
\eq{
\mu_{(d+1)^2}(D_{\llbracket 1, d+1\rrbracket}^{d+1} (N)) \leq C \qnu^{-N}. 
} 

Let $D = D_{\llbracket 1, d+1\rrbracket}^{d+1} (N)$. Define a partition $\{R_\ell\}_{\ell \geq 0}$ of $D$ by
\eq{
R_\ell = \{ \bs \in D  :  \|  \bs^1 \wedge \dots \wedge \bs^d \|  \ \ = \ \qnu^{-\ell} \},
} where $\bs^j$ is the $j$-th row vector of $\bs$.
Denote by $\phi = \phi(d, d+1, d+1) $ and $\psi = \psi(d, d+1, d+1) $ the $\cO$-linear maps as in \textbf{Step 1}, that is, $\psi$ is the projection to the $(d+1)$-th row, and $\phi$ is the projection to the remaining rows. 
Then it follows from Fubini's theorem that
\eqlabel{Eq_Step3_1}{
\begin{split}
	&\mu_{(d+1)^2}(R_\ell) = \int_{(\phi \times \psi)^{-1}(\phi(R_\ell) \times \cOnu^{d+1})} \mathds{1}_{R_\ell}(\bs) d\mu_{(d+1)^2}(\bs) \\
	&= \int_{\phi(R_\ell)} \int_{\cOnu^{d+1}} \mathds{1}_{R_\ell}( (\phi \times \psi)^{-1}(\bs_1, \bs_2) ) d\mu_{d+1}(\bs_2) d\mu_{d(d+1)}(\bs_1).	
\end{split}
}
For $\bs \in \phi(R_\ell)$ and $\bv \in \cOnu^{d+1}$, 
let $\Mllv$ be the $(d+1)\times (d+1)$ matrix with rows $\bs^1, \dots, \bs^d$ and $\bv$.
Consider the following set
\eq{
	\Xpt{\llbracket 1, d\rrbracket} := \{ \bv \in \cOnu^{d+1} :  |  \text{det} \Mllv |  = \ \qnu^{-N } \}.
} 
Similarly to \eqref{Eq_X_Rep}, for any $\bs_1 \in \phi(R_\ell)$ and $\bs_2 \in \cOnu^{d+1}$, we have
\eqlabel{Eq_Step3_2}{
	\mathds{1}_{R_\ell}( (\phi \times \psi)^{-1}(\bs_1, \bs_2) ) \leq \mathds{1}_{X_{\llbracket 1, d\rrbracket}(\bs_1)}(\bs_2).	
}
For $j \in \llbracket 1, d+1 \rrbracket$, let $\tau_{(d+1) , \;  j} (\bs)$ be the $(d+1, j)$-cofactor of $\Mllv$. 
Since $\bs \in \phi(R_\ell)$, we have $\max\limits_{1 \leq j \leq d+1}(|\tau_{(d+1) , \;  j}(\bs)|) =  \| \bs^1 \wedge \dots \wedge \bs^d \| = \qnu^{-\ell}$. 
Note that for any nonzero vector $\ba = (a_j)_{1 \leq j \leq d+1} \in \cOnu^{d+1} 
$ such that $\| \ba \| = |a_{j_0}|$ for some $j_0 \in \llbracket 1, d+1 \rrbracket$, if we denote by $(v_j)_{1 
\leq j \leq d+1}$ the coordinates on $\cOnu^{d+1}$ and set $\bv' = (v_1, \dots, \hat{v_{j_0}}, \dots, v_{d+1})$, then by Fubini's theorem we can write
\eq{
\begin{split}
&\mu_{d+1}(\{(v_j)_{1 
\leq j \leq d+1} \in \cOnu^{d+1} : | 
\sum_{j=1}^{d+1} a_j v_j| \leq \qnu^{-N} \})  \\
&= \int_{\cOnu^d} \int_{\cOnu} \mathds{1}_{\{v_{j_0} \, :\, | a_{j_0} v_{j_0} + \sum\limits_{j \ne j_0} a_j v_j| \leq \qnu^{-N} \}}(v_{j_0}) \, d\mu_1(v_{j_0}) d\mu_d (\bv').
\end{split} 
}
For each fixed $\bv' \in \cOnu^{d}$, the inner integral equals the measure of
\eq{
    B_{\bv'} = \{v_{j_0} \in \cOnu : | a_{j_0} v_{j_0} + \sum\limits_{j \ne j_0} a_j v_j| \leq \qnu^{-N} \}
} which is either the whole $\cOnu$ or a ball of radius $|a_{j_0}|^{-1} \qnu^{-N}$ centered at $-\sum_{j\ne j_0} a_{j_0}^{-1}a_j v_j$. Since $\mu_1(B_{\bv'}) \leq |a_{j_0}|^{-1} \qnu^{-N}$, integrating over $\bv'$ gives
\eq{
\begin{split}
&\mu_{d+1}(\{(v_j)_{1 
\leq j \leq d+1} \in \cOnu^{d+1} : | 
\sum_{j=1}^{d+1} a_j v_j| \leq \qnu^{-N} \})  \\
&\leq \int_{\cOnu^d} \frac{\qnu^{-N}}{|a_{j_0}|} d\mu_d (\bv') = \frac{\qnu^{-N}}{\| \ba \|}.
\end{split}
}

Hence, by the cofactor expansion of $\det \Mllv$ along $(d+1)$-th row, we have
\eqlabel{Eq_Strip}{
\begin{split}
 \mu_{d+1}(\Xball{\llbracket 1, d\rrbracket}) 
	&= \mu_{d+1}(\{ (v_j) \in \cOnu^{d+1} :  | \sum_{j=1}^{d+1} \tau_{(d+1) , \;  j} (\bs) v_j | = \qnu^{-N} \}) \\
	&\leq \frac{\qnu^{-N }}{\max\limits_{1 \leq j \leq d+1}(|\tau_{(d+1) , \;  j} (\bs)|)} = \qnu^{-N + \ell}.
\end{split}
} 
Therefore, it follows from \eqref{Eq_Step3_1}, \eqref{Eq_Step3_2}, and \eqref{Eq_Strip} that
\eqlabel{Eq_Section_Bound}{
\begin{split}
 \mu_{(d+1)^2}(R_\ell) 
	&\leq \int_{\phi(R_\ell)} \mu_{d+1}(X_{\llbracket 1, d\rrbracket}(\bs_1) ) d\mu_{d(d+1)}(\bs_1) \\ 
	&\leq \qnu^{-N+\ell} \mu_{d(d+1)}(\phi(R_\ell)). 
\end{split}
}
Observe that $\phi(R_\ell) \subset \{{}^{t}\!A : A \in E_{\ell}(d+1, d) \}$, where ${}^{t}\!A$ denotes the transpose of $A$. By the induction hypothesis on $(d+1, d)$ of \textbf{Step 3} and by \textbf{Step 2} for $(d+1, d)$, we have
\eqlabel{Eq_Core}{
\begin{split}
\mu_{d(d+1)} (\phi(R_\ell)) 
&\leq \mu_{(d+1)d} (E_{\ell}(d+1, d)) \\
&\leq C_2(d+1, d) \qnu^{-\ell((d+1) -d + 1)} 
= C_2(d+1, d) \qnu^{-2 \ell}.
\end{split}
} Therefore, for any $N \geq 0$, it follows from \eqref{Eq_Section_Bound} and \eqref{Eq_Core} that
\begin{align} \nonumber
	\mu_{(d+1)^2}(D)
	&= \sum_{\ell = 0}^{\infty} \mu_{(d+1)^2}(R_\ell) \leq \sum_{\ell = 0}^{\infty} \qnu^{-N + \ell } \mu_{d(d+1)} (\phi(R_\ell)) \nonumber \\
	&\leq \sum_{\ell = 0}^{\infty} \qnu^{-N + \ell } C_2(d+1, d) \qnu^{-2 \ell} = C_2(d+1, d) (1 - \qnu^{-1})^{-1} \qnu^{-N }, \nonumber
\end{align}
which completes \textbf{Step 3} with $ C(d+1, d+1) = C_2(d+1, d) (1 - \qnu^{-1})^{-1}$.
\end{proof}

Now, we will prove the analogs of Lemma 3.2 and Lemma 3.4 in \cite{KKLM}. Recall from Section \ref{sec2} that the restriction measure $\overline{\mu_{d}}$ of the Haar measure is a $\SL_d(\cOnu)$-invariant probability measure on $\cOnu^d$.
\begin{lem}\label{LemWedgeExp}
	Let $i\in \llbracket 1, d \rrbracket$ and $\bx_1,\dots, \bx_i$ be independent, identically distributed $\SL_d(\cOnu)$-invariant random variables on ${\mathbf K_{\nu}^{d}}$ with law $\overline{\mu_{d}}$. Then
\eq{
\bE(\|\bx_1 \wedge \cdots\wedge \bx_i\|^{-\beta})<\infty
} for all $\beta<d-i+1$.
\end{lem}
\begin{proof}
Recall that $\bs_j$ is the $j$-th column vector of $\bs \in M_{d, i}(\cOnu)$ for every $j \in \llbracket 1, i \rrbracket$. We will compute
\begin{align} \nonumber
\bE(\|\bx_1 \wedge \cdots \wedge \bx_i\|^{-\beta}) & = \int_{M_{d, i}({\mathbf K_{\nu}})}  \|  \bigwedge_{j = 1}^{i}
\bs_j \|^{-\beta} \ 
d\overline{\mu_{di}}(\bs) \\ \nonumber
&= \int_{M_{d, i}(\cOnu)}  \|  \bigwedge_{j = 1}^{i}
\bs_j \|^{-\beta} \ d\mu_{di}(\bs).
\end{align}
Since the norm takes values in $\qnu^\bZ \cup \{ 0\}$, consider the partition  $\{E_{\ell}\}_{\ell \geq 0}$ of $M_{d, i}(\cOnu)$, where 
\eqlabel{Eq_Basic_Decomp}{
E_{\ell} = E_{\ell}(d, i) =  \{ \bs \in M_{d, i}(\cOnu)  :  \|  \bigwedge_{j = 1}^{i}
\bs_j  \|   =  \qnu^{-\ell}
 \}.
}
By Theorem \ref{ThmWedgeUpperBound}, there exists a constant $C_2 = C_2 (d) > 0$ such that $\mu_{di}(E_{\ell}) \leq C_2 \qnu^{-\ell(d-i+1)}$ for any $\ell \geq 0$. Hence, we have
\begin{align}
	\bE(\|\bigwedge_{j=1}^{i} \bx_j \|^{-\beta})
	&= \sum_{\ell = 0}^{\infty} \int_{E_{\ell}}  \|  \bigwedge_{j = 1}^{i}
	\bs_j  \|^{-\beta} \ d\mu_{di}(\bs) 
	= \sum_{\ell = 0}^{\infty} \int_{E_{\ell}}  (\qnu^{-\ell})^{-\beta} \ d\mu_{di}(\bs) \nonumber \\
	&= \sum_{\ell = 0}^{\infty} \qnu^{\ell\beta} \mu_{di}(E_{\ell}) \leq C_2 \sum_{\ell = 0}^{\infty} \qnu^{\ell\beta} \qnu^{-\ell(d-i+1)}, \nonumber
\end{align}
which is finite for any $\beta<d-i+1$.
\end{proof}

\begin{lem}\label{LemWedgeErrExp}
Let $\bx_1, \dots, \bx_i$ be random variables as in Lemma \ref{LemWedgeExp} and $C_2 = C_2 (d)$ the constant as in Theorem \ref{ThmWedgeUpperBound}. Then for every $\kappa \in \bNN$, we have
\eq{
\bE(\|\bx_1 \wedge \dots \wedge \bx_i \|^{-(d-i+1)} \mathds{1}(\|\bx_1 \wedge \dots \wedge \bx_i \| > \qnu^{-\kappa})) \leq C_2 \kappa.
}
\end{lem}
\begin{proof}
Recall $E_\ell$ as in \eqref{Eq_Basic_Decomp}. Since $\mu_{di}(E_{\ell}) \leq C_2 \qnu^{-\ell(d-i+1)}$ and
\eq{
	\{ \bs \in M_{d, i}(\cOnu)  :  \|  \bigwedge_{j = 1}^{i}
	\bs_j  \|  > \qnu^{-\kappa}
	\} = \bigcup_{\ell = 0}^{\kappa - 1} E_\ell,
} we have
\eq{
	\begin{split}
		&\bE(\| \bigwedge_{j = 1}^{i} \bx_j \|^{-(d-i+1)} \mathds{1}(\| \bigwedge_{j = 1}^{i} \bx_j \| > \qnu^{-\kappa})) \ \\
		& = \int_{M_{d, i}({\mathbf K_{\nu}})}  \|  \bigwedge_{j = 1}^{i}
		\bs_j  \|^{-(d-i+1)} \mathds{1}(\|\bigwedge_{j = 1}^{i}
		\bs_j \| > \qnu^{-\kappa})  \ 
d\overline{\mu_{di}}(\bs) \nonumber \\
		&= \int_{M_{d, i}(\cOnu)}  \|  \bigwedge_{j = 1}^{i}
		\bs_j  \|^{-(d-i+1)} \mathds{1}(\|\bigwedge_{j = 1}^{i}
		\bs_j \| > \qnu^{-\kappa})  \ d\mu_{di}(\bs) \nonumber \\
        &= \sum_{\ell = {0}}^{\kappa - 1} \int_{E_{\ell}}  \|  \bigwedge_{j = 1}^{i}
		\bs_j  \|^{-(d-i+1)} \ d\mu_{di}(\bs) \nonumber \\
		&= \sum_{\ell = {0}}^{\kappa - 1} \int_{E_{\ell}}  {(\qnu^{-\ell})}^{-(d-i+1)} \ d\mu_{di}(\bs) \nonumber \\
		&= \sum_{\ell = {0}}^{\kappa - 1} \qnu^{\ell(d-i+1)}  \mu_{di}(E_\ell) \leq \sum_{\ell = {0}}^{\kappa - 1} \qnu^{\ell(d-i+1)}  C_2 \qnu^{-\ell(d-i+1)} \leq C_2 \kappa. \nonumber
	\end{split}
}

\end{proof}
Now we obtain an analog of \cite[Proposition 3.1]{KKLM} in the global function fields, but only for wedge vectors of small dimensions. Let $\lam$ be the normalized Haar measure on the compact group $\SL_{m+n}(\cOnu)$.
\begin{prop}\label{OrtEst}
For $i\in \llbracket 1, m \rrbracket$, let $\beta_{i} = \frac{m}{i}$.
Then there exists a constant $C > 0$ depending only on $m,n$ such that 
\eq{
\int_{\SL_{m+n}(\cOnu)}\|g_{t} \bk \bv\|^{-\beta_i}d\lam(\bk) \leq Ct\qnu^{-mnt}\|\bv\|^{-\beta_i}
} for any $t \in \bNN$ and any decomposable $i$-vector $\bv \in\bigwedge^{i} \mathbf K_\nu^{m+n}$.
\end{prop}
\begin{proof}
Set
\eq{
I_{t}(\bv) = \int_{\SL_{m+n}(\cOnu)} \ \|g_{t} \bk \bv\|^{-\beta_i}d\lam(\bk)
} for each $t \in \bNN$ and decomposable $i$-vector $\bv\in\bigwedge^{i}\mathbf K_\nu^{m+n}$. 
By \cite[Corollary 2.11]{Bang} (see also \cite[Lemma 2.7]{HP} for the transitivity on Grassmanian), for any decomposable $i$-vectors $\bv, \bv'$ satisfying $\|\bv\| = \|\bv'\|$, there exist $c \in \mathbf K_\nu$ with $|c|=1$ and $\bg \in \SL_{m+n}(\cOnu)$ such that $\bv=c \bg \bv'$. Using the $\SL_{m+n}(\cOnu)$-invariance of $\lambda$ and the homogeneity of the norm, we have
\eq{
\begin{split}
I_{t}(\bv)
&=\int_{\SL_{m+n}(\cOnu)} \ \|g_{t} \bk \bv\|^{-\beta_i} d\lam(\bk) 
=\int_{\SL_{m+n}(\cOnu)} \ \|g_{t} \bk (c \bg \bv')\|^{-\beta_i} d\lam(\bk) \\
&=\int_{\SL_{m+n}(\cOnu)} \ \|g_{t} \bk \bg \bv'\|^{-\beta_i} d\lam(\bk)
=\int_{\SL_{m+n}(\cOnu)} \ \|g_{t} \bk \bv'\|^{-\beta_i} d\lam(\bk) \\
&=I_{t}(\bv')
\end{split}
} Thus, for any decomposable $i$-vectors $\bv, \bw \in \bigwedge^{i}\mathbf K_\nu^{m+n}$, we have
\eq{
\|\bv\|^{\beta_i} I_{t}(\bv) = I_{t}(\frac{\bv}{\|\bv\|}) = I_{t}(\frac{\bw}{\|\bw\|}) = \|\bw\|^{\beta_i} I_{t}(\bw).
}

Hence, we have
\eqlabel{Eq_It}{
I_{t}(\bv)=C(t)\| \bv \|^{-\beta_i}
} for some function $C(t)$.

Let $\bv=\bx_1 \wedge \dots \wedge \bx_i$, where $\bx_1, \dots, \bx_i$ are independent random variables on $\mathbf K_{\nu}^{m+n}$ whose distribution is $\overline{\mu_{m+n}}$. 
Since $\beta_i = \frac{m}{i} < m + n - i + 1$ for every $i \leq m$, we obtain $ C'_1 := \bE(\|\bx_1 \wedge \cdots\wedge \bx_i\|^{-\beta_i})<\infty$ by Lemma \ref{LemWedgeExp}. 
Taking the expectation on \eqref{Eq_It}, and since the distributions of $\bx_1, \dots, \bx_i$ are $\SL_{m+n}(\cOnu)$-invariant, we obtain
\eq{
C(t) = \frac{1}{C'_1}\bE (I_{t}(\bx_1 \wedge \cdots \wedge \bx_i)) = \frac{1}{C'_1}\bE(\|{g}_{t}(\bx_1 \wedge \cdots \wedge \bx_i) \|^{-\beta_i}).
} Hence, it suffices to show that there exists a constant $c$ depending only on $m$ and $n$ such that
\eqlabel{Eq_DiagWedgeEst}{
\bE(\|{g}_{t}(\bx_1 \wedge \cdots \wedge \bx_i) \|^{-\beta_i}) \leq ct\qnu^{-mnt}
} for all $t \in \bNN$.

Let $V_u$ and $V_s$ be the subspaces of $\mathbf K_\nu^{m+n}$ spanned by $\{\be_j\}_{1\leq j \leq m}$ and $\{\be_{m+j}\}_{1 \leq j \leq n}$, respectively. Let $p_{u}^{(1)} : \mathbf K_\nu^{m+n} \rightarrow {V_u}$ and $p_{s}^{(1)} : \mathbf K_\nu^{m+n} \rightarrow {V_s}$ be the projections corresponding to the direct sum $\mathbf K_{\nu}^{m+n} = V_u \bigoplus V_s$. Define the projection $p_{u}^{(i)} : \bigwedge^{i}  \mathbf K_\nu^{m+n} \rightarrow \bigwedge^{i}  {V_u} $ by
\eq{
	p_{u}^{(i)}(\be_I) = \bigwedge_{k=1}^{i} p_{u}^{(1)}(\be_{\iota_k})
} for each $I = \{ \iota_1, \dots, \iota_i \} \in \wp^{m+n}_{i}$ and extending it linearly. Define the projection $p_{s}^{(i)} : \bigwedge^{i}  {\mathbf K_\nu}^{m+n} \rightarrow \bigwedge^{i} {V_s}$ in a similar way.
Note that
\eqlabel{Eq_Proj_Decomp}{
	{p_{u}^{(i)}(\bv_{\llbracket 1, i \rrbracket}) = \bigwedge_{j=1}^{i} p_{u}^{(1)}(\bv_j)}
} for any $\bv_1, \dots, \bv_i \in \mathbf K_{\nu}^{m+n}$.

For $1 < i < m$, since $p_{u}^{(i)}$ is a projection, we have
\eqlabel{Eq_Proj_Wedge}{
	\| p_{u}^{(i)} g_t(\bigwedge_{j=1}^{i} \bx_j) \| = \| p_{u}^{(i)}( \bigwedge_{j=1}^{i} g_t \bx_j) \| \leq \| g_t (\bigwedge_{j=1}^{i} \bx_j) \|.
}
Note that $p_{u}^{(1)}(\bx_1), \dots$, $p_{u}^{(1)}(\bx_i)$ are independent random variables on $\mathbf K_\nu^m$ whose distribution is $\overline{\mu_m}$. Thus, applying Lemma \ref{LemWedgeExp} with $d = m$, we obtain
\eqlabel{Eq_Proj_Finite}{
	B_i := \bE(\|p_{u}^{(1)}(\bx_1) \wedge \cdots \wedge p_{u}^{(1)}(\bx_i)\|^{-\beta_i})<\infty
} since $\beta_i = \frac{m}{i} < m -i + 1$ for $1 < i < m$.
Hence, it follows from \eqref{Eq_Proj_Decomp}, \eqref{Eq_Proj_Wedge}, \eqref{Eq_Proj_Finite} and the equality $p_{u}^{(1)}(g_t \bx_j) = \pinu^{-nt}p_{u}^{(1)}( \bx_j)$ that
\eqlabel{Eq_Res1}{
\begin{split}    
 \bE(\| g_t(\bigwedge_{j=1}^{i} \bx_j )\|^{-\beta_i})
	&\leq \bE(\| \bigwedge_{j=1}^{i} p_{u}^{(1)}(g_t \bx_j) \|^{-\beta_i}) \\
	&= \bE(\qnu^{-ni \beta_i t}\| \bigwedge_{j=1}^{i} p_{u}^{(1)}(\bx_j) \|^{-\beta_i})  \leq \qnu^{-mnt} \max_{1 < i < m}B_i, 
\end{split}
}
which satisfies \eqref{Eq_DiagWedgeEst} since $t \geq 1$ and $\max_{1 < i < m}B_i$ depends only on $m$.

Now, assume $i=1$ or $i=m$. First, by \eqref{Eq_Proj_Wedge} and since $g_t$ dilates the norm of every vector at least by $\qnu^{-mt}$, observe that
\eq{
	\|g_t (\bigwedge_{j=1}^{i} \bx_j)\|
	\geq \max(
		\| g_t p_{u}^{(i)} (\bigwedge_{j=1}^{i} \bx_j) \|,
		\qnu^{-imt} \| \bigwedge_{j=1}^{i} \bx_j \|
	).
} Thus, for any $\kappa \in \bNN$, we can separate the left hand side of \eqref{Eq_DiagWedgeEst} into
\eqlabel{Eq_DiagWedgeSep}{
\bE(\|{g}_{t}(\bigwedge_{j=1}^{i} \bx_j ) \|^{-\beta_i}) \leq E_1 + E_2,
} where
\eq{
	\begin{split}
		E_1 &= \bE(\|{g}_{t}p_{u}^{(i)}(\bigwedge_{j=1}^{i} \bx_j ) \|^{-\beta_i} \mathds{1}(\|p_{u}^{(i)} (\bigwedge_{j=1}^{i} \bx_j ) \| > \qnu^{-\kappa t})); \\
		E_2 &= \qnu^{im{\beta}_{i}t} \bE(\|\bigwedge_{j=1}^{i} \bx_j  \|^{-\beta_i} \mathds{1}(\|p_{u}^{(i)} (\bigwedge_{j=1}^{i} \bx_j ) \| \leq \qnu^{-\kappa t})).
	\end{split}
}
By Lemma \ref{LemWedgeErrExp} for $d = m$ with random variables $p_{u}^{(1)}(\bx_1), \dots, p_{u}^{(1)}(\bx_i)$ and by \eqref{Eq_Proj_Decomp} we have
\eqlabel{Eq_ProjWedgeEst}{
	\bE(\|p_{u}^{(i)}(\bigwedge_{j=1}^{i} \bx_j ) \|^{-(m-i+1)} \mathds{1}(\|p_{u}^{(i)} (\bigwedge_{j=1}^{i} \bx_j ) \| > \qnu^{-\kappa t})) \leq C_2 \kappa t
}
Since $\beta_i = \frac{m}{i}= m-i+1$ and $g_t p_{u}^{(1)}(\bx_j) = \pinu^{-nt} p_{u}^{(1)}(\bx_j) $ for $j \in \llbracket 1, i \rrbracket$, it follows from \eqref{Eq_ProjWedgeEst} that 
\eqlabel{Eq_E_1_Res}{
\begin{split}    
 E_1 &= \bE(\| \bigwedge_{j=1}^{i} {g}_{t}p_{u}^{(1)}(\bx_j) \|^{-\beta_i} \mathds{1}(\| p_{u}^{(i)} (\bigwedge_{j=1}^{i} \bx_j ) \| > \qnu^{-\kappa t})) \\
	&= \qnu^{-ni \beta_i t} \bE(\| \bigwedge_{j=1}^{i} p_{u}^{(1)}(\bx_j)  \|^{-\beta_i} \mathds{1}(\|p_{u}^{(i)} (\bigwedge_{j=1}^{i} \bx_j) \| > \qnu^{-\kappa t})) \\ 
	&\leq \qnu^{-ni \beta_i t} (C_2 \kappa t) = C_2 \kappa t \qnu^{-mnt}. 
\end{split}
}
To estimate $E_2$, we use a direct calculation applying Theorem \ref{ThmWedgeUpperBound}, which is a different method from \cite[Proposition 3.1]{KKLM} using conditional expectation. First, assume $i=1$. We have
\eqlabel{Eq_E_2_bound1}{
\begin{split}    
 E_2 &= \qnu^{m^{2}t} \bE(\|\bx_1  \|^{-m} \mathds{1}(\|p_{u}^{(1)} (\bx_1 ) \| \leq \qnu^{-\kappa t})) \\
	&= \qnu^{m^{2}t} \int_{{\textbf K_\nu}^{m+n}}  \| \bs \|^{-m} \mathds{1}(\| p_{u}^{(1)}(\bs) \| \leq \qnu^{-\kappa t})  d\overline{\mu}_{m+n}(\bs) \\
    &= \qnu^{m^{2}t} \int_{\cOnu^{m+n}}  \| \bs \|^{-m} \mathds{1}(\| p_{u}^{(1)}(\bs) \| \leq \qnu^{-\kappa t})  d\mu_{m+n}(\bs) \\
	&= \qnu^{m^{2}t} \sum_{\ell=0}^{\infty} \int_{\cOnu^{m+n}}  (\qnu^{-\ell})^{-m} \mathds{1}(\| \bs \| = \qnu^{-\ell},  \| p_{u}^{(1)}(\bs) \| \leq \qnu^{-\kappa t})  d\mu_{m+n}(\bs) \\
	&= \qnu^{m^{2}t} \sum_{\ell=0}^{\infty} \qnu^{m\ell} \mu_{m+n}(\{\bs \in \cOnu^{m+n} : \| \bs \| = \qnu^{-\ell},  \|p_{u}^{(1)}(\bs)\| \leq \qnu^{-\kappa t}\}). 
\end{split}
}
Let $B_{\ell, \kappa t} = \mu_{m+n}(\{\bs \in \cOnu^{m+n} : \| \bs \| = \qnu^{-\ell},  \|p_{u}^{(1)}(\bs)\| \leq \qnu^{-\kappa t}\})$.
 If $\ell \geq \kappa t$, since $\qnu^{-\ell} \leq \qnu^{- \kappa t}$ and $\|p_{u}^{(1)}(\bs)\| \leq \|\bs\|$, we have
\eqlabel{Eq_E_2_bound2}{
	B_{\ell, \kappa t} \leq \mu_{m+n}(\{\bs \in \cOnu^{m+n} : \|\bs\| \leq \qnu^{-\ell} \}) = \qnu^{-(m+n)\ell}.
}
If $\ell < \kappa t$, since $\|\bs\| = \max(\|p_{u}^{(1)}(\bs)\|, \|p_{s}^{(1)}(\bs)\|)$ for any $s \in \cOnu^{m+n}$, we have
\eqlabel{Eq_E_2_bound3}{
\begin{split}
	B_{\ell, \kappa t} &= \mu_{m+n}(\{\bs : \| p_{s}^{(1)}(\bs)\| = q^{-\ell},  \| p_{u}^{(1)}(\bs) \| \leq q^{-\kappa t}\}) 
 \leq q^{-m\kappa t -n\ell}.
\end{split}
}
Therefore, it follows from \eqref{Eq_E_2_bound1}, \eqref{Eq_E_2_bound2} and \eqref{Eq_E_2_bound3} that
\eqlabel{Eq_E_2_bound4}{
\begin{split}
	E_2 &\leq \qnu^{m^{2}t} \Bigl( \sum_{\ell=0}^{\kappa t - 1}\qnu^{m\ell} (\qnu^{-n\ell-m\kappa t}) + \sum_{\ell=\kappa t}^{\infty}\qnu^{m\ell}(\qnu^{-(m+n)\ell}) \Bigr) \\
	&= \qnu^{m^{2}t} \Bigl( \qnu^{-m\kappa t} \sum_{\ell=0}^{\kappa t - 1}\qnu^{(m-n)\ell}  + \sum_{\ell=\kappa t}^{\infty}\qnu^{-n\ell} \Bigr) \\
	&= \qnu^{m^{2}t} \Bigl( \qnu^{-m\kappa t} \sum_{\ell=0}^{\kappa t - 1}\qnu^{(m-n)\ell}  + \frac{\qnu^{-n \kappa t}}{1 - \qnu^{-n}} \Bigr).
\end{split}
}
Note that 
\eq{
 \sum_{\ell=0}^{\kappa t - 1}\qnu^{(m-n)\ell}
 \leq \begin{cases}
 	\frac{\qnu^{(m-n)\kappa t} - 1}{\qnu^{(m-n)} - 1} \leq \frac{\qnu^{(m-n)\kappa t}}{1 - \qnu^{-n}}     & \quad \text{if } m > n \\
 	\kappa t \ \ \ \ \ \ \ \ \ \ \leq \frac{\kappa t}{1 - \qnu^{-n}}     & \quad \text{if } m \leq n.
 \end{cases}
} Hence, it follows from \eqref{Eq_E_2_bound4} that
\eq{
	E_2 \leq
\begin{cases}
	(1 - \qnu^{-n})^{-1}\qnu^{m^2 t}(2\qnu^{-n \kappa t})   & \quad \text{if } m > n \\
	(1 - \qnu^{-n})^{-1}\qnu^{m^2 t}(\kappa t \qnu^{-m \kappa t} + \qnu^{-n \kappa t})    & \quad \text{if } m \leq n,
\end{cases}
} which implies 
\eqlabel{Eq_E_2_Res1}{
	E_2 \leq 2(1 - \qnu^{-n})^{-1} \kappa t \qnu^{-mnt}
} for any $m, n \in \bNN$ and $\kappa$ $>$ $\max(\frac{m(m+n)}{m}, \frac{m(m+n)}{n})$.

Now assume $i = m$. Define
\eq{
	Z = \{\bs \in M_{m+n, m} (\cOnu) : \| p_{u}^{(m)}( \bigwedge_{j = 1}^{m}
	\bs_j ) \| \leq \qnu^{-\kappa t} \}.
} Consider a partition $\{R_\ell\}_{\ell \geq 0}$ of $M_{m+n, m}(\cOnu)$ defined by
\eq{
	R_\ell = \{\bs \in M_{m+n, m} (\cOnu) : \|  \bigwedge_{j = 1}^{m}
	\bs_j  \| = \qnu^{-\ell} \}.
}
Then, we have
\eqlabel{Eq_E_2_Decomp1}{
     E_2 = \qnu^{m^2 t} \int_{Z}  \|  \bigwedge_{j=1}^{m} \bs_j  \|^{-1} d\mu_{(m+n)m}(\bs)  = \qnu^{m^2 t} \sum_{\ell=0}^{\infty} \qnu^{\ell} \mu_{(m+n)m}(R_\ell \cap Z). 
}
Note that $R_\ell \subset E_\ell (m+n, m)$ and $R_\ell \cap Z \subset F_{\ell, \kappa t} (m+n, m)$ which are defined in Theorem \ref{ThmWedgeUpperBound}; indeed, the latter inclusion follows from \eqref{Eq_Proj_Decomp}, since
$p_{u}^{(m)}( \bigwedge_{j = 1}^{m}	\bs_j )
= \bigwedge_{j = 1}^{m}p_{u}^{(1)}(\bs_j)
= c_{\llbracket 1, m\rrbracket}(\bigwedge_{j = 1}^{m}	\bs_j)\be_{\llbracket 1, m\rrbracket}
$.
 If $\ell \geq \kappa t$, we have $R_\ell \subseteq Z$ since $\|p_u^{(m)}(\bv)\| \leq \|\bv\|$ for any $\bv \in \bigwedge^{m}\mathbf K_\nu^{m+n}$. Hence, by Theorem \ref{ThmWedgeUpperBound}, there exists a constant $C'_2 = C'_2(m,n)> 0$ such that
\eqlabel{Eq_E_2_bound5}{
	\mu_{(m+n)m}(R_\ell \cap Z) \leq \mu_{(m+n)m}(E_\ell (m+n, m)) \leq C'_2 \qnu^{-\ell (n + 1) }.
}
If $\ell < \kappa t$, by Theorem \ref{ThmWedgeUpperBound} again, there exists a constant $C_1 = C_1(m,n) > 0$ such that
\eqlabel{Eq_E_2_bound6}{
\begin{split}    
 \mu_{(m+n)m}(R_\ell \cap Z) &\leq \mu_{(m+n)m}(F_{\ell, \kappa t} (m+n, m)) \\
	&\leq C_1 \qnu^{- \kappa t -n\ell} \binom{\kappa t -\ell + \binom{m+n}{m}}{\kappa t -\ell + 1}. 
\end{split}
}
Note that given $a \in \bNN$, there exists a constant $C(a) > 0$ depending only on $a$ such that
\eq{
\binom{x+a}{x+1} \leq \frac{(x+a)^{a-1}}{(a-1)!} \leq C(a) \qnu^{\frac{1}{2}x }
} for any $x \in \bNN$ since the exponential growth is faster than the polynomial one. Hence, there exists a constant $C_3 > 0$ depending only on $m$ and $n$ such that
\eq{
	\binom{\kappa t - \ell + \binom{m+n}{m}}{\kappa t - \ell + 1} \leq C_3 \qnu^{\frac{1}{2} (\kappa t - \ell)}
} for any $0 \leq \ell \leq \kappa t - 1$. 

Therefore, it follows from \eqref{Eq_E_2_Decomp1}, \eqref{Eq_E_2_bound5} and \eqref{Eq_E_2_bound6} that
\eqlabel{Eq_E_2_bound7}{
	\begin{split}
	E_2 &\leq \qnu^{m^2 t} (\sum_{\ell=0}^{\kappa t - 1} \qnu^{\ell} (C_1 \qnu^{- \kappa t -n\ell} (C_3 \qnu^{\frac{1}{2} (\kappa t - \ell)})) + \sum_{\ell=\kappa t}^{\infty} \qnu^{\ell} (C'_2 \qnu^{-\ell (n + 1) })) \\
	&= \qnu^{m^2 t} (C_1 C_3 \qnu^{-\frac{1}{2} \kappa t} \sum_{\ell=0}^{\kappa t - 1} \qnu^{- (n + \frac{1}{2})\ell} + C'_2 \sum_{\ell=\kappa t}^{\infty} \qnu^{-n\ell}) \nonumber \\
	&= \qnu^{m^2 t} (C_1 C_3 \qnu^{-\frac{1}{2} \kappa t} \frac{1 - \qnu^{-(n + \frac{1}{2})\kappa t }}{1 - \qnu^{-(n + \frac{1}{2})}} + C'_2 \frac{\qnu^{-n \kappa t}}{1 - \qnu^{-n}}) \nonumber \\
	&\leq \qnu^{m^2 t} (C_1 C_3 \qnu^{-\frac{1}{2} \kappa t} \frac{1}{1 - \qnu^{-(n + \frac{1}{2})}} + C'_2 \frac{\qnu^{-n \kappa t}}{1 - \qnu^{-n}}) \ \ \ \ \ \ \ \ \ \ \ \ \ \ \ \ \ \ \ \ \ \ \nonumber \\
	&\leq C_4 \qnu^{m^2 t}(\qnu^{-\frac{1}{2} \kappa t} + \qnu^{-n \kappa t}  ), \nonumber
\end{split}
} where $C_4 = \max(C_1 C_3 (1 - \qnu^{-(n + \frac{1}{2})})^{-1}, C'_2 (1 - \qnu^{-n})^{-1})$. Hence, we have
\eqlabel{Eq_E_2_Res2}{
	E_2 \leq C_4 \qnu^{-mnt}
} for any $\kappa \geq 2(m^2 + mn)$.

As a summary, it follows from \eqref{Eq_DiagWedgeSep}, \eqref{Eq_E_1_Res}, \eqref{Eq_E_2_Res1} and \eqref{Eq_E_2_Res2} that when $i = 1$ or $i = m$, we have
\eq{
\bE(\|{g}_{t}(\bx_1 \wedge \cdots \wedge \bx_i) \|^{-\beta_i}) \leq (C_1 \kappa t  + \max(2(1-\qnu^{-n})^{-1} \kappa t ,  C_4))\qnu^{-mnt}
} for any $t \in \bNN$ and any $\kappa > 2(m^2 + mn)$. Thus, by fixing some $\kappa > 2(m^2 + mn)$, we obtain
\eqlabel{Eq_Res2}{
\bE(\|{g}_{t}(\bx_1 \wedge \cdots \wedge \bx_i) \|^{-\beta_i}) \leq C t \qnu^{-mnt}
} for all $t \in \bNN$, where $C$ is a constant depending on $m$ and $n$.

Therefore, from \eqref{Eq_Res1} and \eqref{Eq_Res2}, we reach the goal \eqref{Eq_DiagWedgeEst}.
\end{proof}

\section{the Hodge duality and the proof of Proposition \ref{LinAlgEst}}\label{Sec4}

We let $d = m+n$ for this section and the remaining sections. While \cite[Proposition 3.1]{KKLM} covers all $1 \leq i \leq d$ using the equivariance of the Hodge duality under $\mathrm{SO}_{d}(\bR)$, the Hodge dual operator is not equivariant under the ultrametric orthogonal group $\SL_{d}(\cOnu)$. Instead, we apply the duality to horospherical subgroups.

Note that the exterior powers and their natural basis can be defined for any free $R$-module with scalar in a commutative ring $R$ with unity. Thus, we can define the standard Hodge dual operator between $\bigwedge^{i} \mathbf K_\nu^{d}$ and $\bigwedge^{d-i} \mathbf K_\nu^{d}$
 in the same way as a linear map, even though there is no inner product structure.
\begin{defi}
 Let $\{\be_j\}_{j \in \llbracket 1, d\rrbracket}$ be the standard basis of $\mathbf K_{\nu}^{d}$. 	For each $i \in \llbracket 1, d \rrbracket$, we define \textit{the Hodge dual operator} $* : \bigwedge^{i}{ \mathbf K_{\nu}^{d}} \rightarrow \bigwedge^{d-i}{\mathbf K_{\nu}^{d}}$ linearly by
\eq{
*(\be_I) =\sigma_{I} \be_{ I^c },
} where $I = \{a_1 < \cdots < a_{i} \} \in \wp_i^d$, $I^c = \llbracket 1, d \rrbracket \smallsetminus I$, and $\sigma_{I} = (-1)^{\frac{i(i+1)}{2} + \sum_{s=1}^{i}{a_s}}$. \end{defi}
Note that the above definition is the same one as in the real case and the term $\sigma_{I}$ comes from $\be_I \wedge \be_{I^c} = \sigma_{I} \be_{\llbracket 1, d \rrbracket}$. The following lemma says that some properties related to exterior powers still hold for the global function fields. 
This property of the Hodge dual operator is well-known, but we cannot find a suitable reference for non-Archimedean spaces.
\begin{lem}\label{lem_dual}
For any $i \in \llbracket 1, d-1 \rrbracket$, $\bv \in \bigwedge^{i}{{\mathbf K_\nu}^{d}}$, and $\bg \in \SL_d(\mathbf K_\nu)$, we have 
\begin{enumerate}
    \item\label{lem_dual_1} $\|*\bv\| = \|\bv\|$;
    \item\label{lem_dual_2} $*(\bg \bv) =  {}^{t}(\bg^{-1}\!) (*\bv)$.
\end{enumerate}
\end{lem}	

\begin{proof}
\eqref{lem_dual_1} Let $\bv = \sum_{I} v_I \be_I\in \bigwedge^{i} \mathbf K_\nu^{d}$ be an $i$-vector. Since
\eq{
*\bv = *(\sum_{I} v_I \be_I ) = \sum_{I} v_I \! * \! (\be_I)  = \sum_{I} v_I \sigma_I \be_{I^{c}}
} and $|\sigma_I | = 1$ for all $I \in \wp_{i}^d$, it follows that
\eq{
	\|*\bv\| = \|\sum_{I} v_I \sigma_I \be_{I^{c}}\| = \max_{I}|v_I \sigma_I| = \max_{I}|v_I| = \|\bv\|.
}
\eqref{lem_dual_2} It is enough to show that for any $\bg \in \SL_{d}(\mathbf K_\nu)$ and $I\in\wp_i^d$, we have
\eqlabel{Eq_Dual_Matrix}{
*(\bg \be_I) = {}^{t}(\bg^{-1}\!) (*\be_I).
}
First, observe that
\eqlabel{Eq_Wedge_Matrix}{
\bg \be_I = \sum_{J\in \wp_i^d} (\det \bg_{J, I}) \be_{J},
} where $\det \bg_{I, J}$ is the $(I,J)$-minor of $\bg$.
Hence, its Hodge dual is
\eq{
	*(\bg \be_I) = *(\sum_{J} (\det \bg_{J, I}) \be_{J}) = \sum_{J} (\det \bg_{J, I}) *(\be_{J}) = \sum_{J} \sigma_J (\det \bg_{J, I}) \be_{J^{c}}.
}
On the other hand, using \eqref{Eq_Wedge_Matrix} with $ {}^{t}(\bg^{-1}\!) $, we have
\begin{align*}
	{}^{t}(\bg^{-1}\!) (*\be_I) &= \sigma_I {}^{t}(\bg^{-1}\!) ( \be_{I^c}) = \sum_{J \in \wp_{d-i}^d} \sigma_I \det({}^{t}(\bg^{-1}\!))_{J, I^{c}} \be_{J} \\
	&= \sum_{J \in \wp_{i}^d} \sigma_I \det({}^{t}(\bg^{-1}\!))_{J^{c}, I^{c}} \be_{J^{c}}. 
\end{align*}
Therefore, to prove \eqref{Eq_Dual_Matrix}, it suffices to show that
 \eqlabel{Eq_Dual_Main}{
	\det \bg_{J, I}  =  \sigma_I \sigma_J \det ({}^{t}(\bg^{-1}\!))_{J^{c}, I^{c}}
} for any $I, J \in \wp_{i}^d$. Since $\det \bg =1$, \eqref{Eq_Dual_Main} is Jacobi's identity, which holds over any commutative ring (see for instance \cite[Lemma A.1]{CSS}).
\end{proof}

\begin{remark} 
The Hodge dual operator preserves the decomposability. Indeed, if $\bw = \bv_1 \wedge \dots \wedge \bv_i$ is a nonzero decomposable vector, 
then $\bv_1, \dots, \bv_i$ are linearly independent over $\mathbf K_\nu$, and thus it can be extended to a basis $\{\bv_1, \dots, \bv_{d}\}$ of $\mathbf K_\nu^{d}$. Let $\bv$ be the $(d,d)$-matrix with columns $\bv_1, \dots, \bv_d$. Then $\bw = \bv\be_{\llbracket 1, i \rrbracket}$ and $\bv$ is invertible. Hence, by Lemma \ref{lem_dual}, we have
\eq{
	\begin{split}
	*\bw &= {}^{t}(\bv^{-1}\!) \sigma_{ \llbracket 1, i \rrbracket } \be_{\llbracket i+1, d \rrbracket} = (\sigma_{\llbracket 1, i \rrbracket} {}^{t}(\bv^{-1}\!) \be_{i+1}) \wedge \cdots \wedge ({}^{t}(\bv^{-1}\!) \be_{d}),
	\end{split}
} which implies that $*\bw$ is decomposable.
\end{remark} 

Now, we are ready to prove Proposition \ref{LinAlgEst} using the Hodge dual operator.
\begin{proof}[Proof of Proposition \ref{LinAlgEst}]
First, let $i \in \llbracket 1, m \rrbracket$ in order to apply Proposition \ref{OrtEst}.
Define $P( \cOnu)$ and $U( \cOnu)$ by
\eq{
\begin{split}
	&P( \cOnu) = \left\{\left(\begin{matrix} \ba & 0 \\ \bb & \bc \end{matrix} \right) :  \ba \in M_{m,m}(\cOnu), \bb \in M_{n,m}(\cOnu), \bc \in M_{n,n}(\cOnu), \det \ba \det \bc =1\right\}; \\
	&U( \cOnu) = \left\{\left(\begin{matrix} \mathbf{I}_m & \bs \\ 0 & \mathbf{I}_n \end{matrix} \right) : \bs \in M_{m,n}(\cOnu) \right\}.
\end{split}
} 
Using the fact that $\cOnu$ is a compact subring of $\mathbf K_\nu$, one can easily show that $P( \cOnu)$ and $U( \cOnu)$ are compact subgroups of $\SL_d(\cOnu)$. Moreover, there exists an open neighborhood of $\mathbf{I}_d$ in $\SL_d(\cOnu)$ contained in $P( \cOnu)U( \cOnu)$.
For instance, if we take
\eq{
V = (\mathbf{I}_{d} + M_{d,d}(\pinu \cOnu)) \cap \SL_d(\cOnu).
} as an open neighborhood of $\mathbf{I}_{d}$ in $\SL_d(\cOnu)$, then $V$ is contained in $P(\cOnu)U(\cOnu)$. 
Indeed, let $\bx \in V$ and write $\bx = \left(\begin{matrix} \mathbf{I}_m + \ba & \bb \\ \bc & \mathbf{I}_n + \bd \end{matrix} \right)$. 
Applying the Leibniz formula for the determinant to $\mathbf{I}_m + \ba$ and $\mathbf{I}_n + \bd$, we have 
\eq{
	\det(\mathbf{I}_m + \ba) - 1, \det(\mathbf{I}_n + \bd) - 1 \in \pinu \cOnu,
} which implies that $|\text{det}(\mathbf{I}_m + \ba)| = |\text{det}(\mathbf{I}_n + \bd)| = 1$, and thus, $\mathbf{I}_m + \ba \in \GL_m(\cOnu)$ and $\mathbf{I}_n + \bd \in \GL_n(\cOnu)$.  Hence, we have the following decomposition
\eq{\begin{split}
	\left(\begin{matrix} \mathbf{I}_m + \ba & \bb \\ \bc & \mathbf{I}_n + \bd \end{matrix} \right)
	&= \left(\begin{matrix} \mathbf{I}_m + \ba & 0 \\ \bc & \mathbf{I}_n + \bd - \bc(\mathbf{I}_m + \ba)^{-1} \bb \end{matrix} \right)
	  \left(\begin{matrix} \mathbf{I}_m & (\mathbf{I}_m + \ba)^{-1} \bb \\ 0 & \mathbf{I}_n \end{matrix} \right)\\[0.5ex]
   &\in P(\cOnu)U(\cOnu).
   \end{split}} 

Since $\SL_d(\cOnu)$, $P(\cOnu)$ and $U(\cOnu)$ are compact groups, they are unimodular. 
Let $\lam_P$ and $\lam_U$ be the normalized Haar measure of $P(\cOnu)$ and $U(\cOnu)$, respectively. 
Since $\mathbf{I}_{d} \in V \subset P(\cOnu)U(\cOnu)$ and $P(\cOnu) \cap U(\cOnu) = \{\mathbf{I}_{d}\}$, 
we have by \cite[Lemma 11.31]{EW} that 
\eqlabel{Eq_Msr_Decomp}{
	\lam |_{P(\cOnu)U(\cOnu)} = \lam (P(\cOnu)U(\cOnu)) \phi_{*}(\lam_P \times \lam_U),
} where $\phi : P(\cOnu) \times U(\cOnu) \rightarrow \SL_d(\cOnu)$ is defined by $\phi(\bp, \bu) = \bp \bu$. Note that $C_\lam := \lam (P(\cOnu)U(\cOnu))$ is a constant depending only on $d$.

Denote by $\|\mathbf{x}\|_{op}$ the operator norm for $\mathbf{x} \in M_{d,d}(\mathbf K_\nu)$ induced by the ultrametric norm in $\mathbf K_\nu^{d}$. 
Note that $\|g_{t} \bp g_{-t}\|_{op} \leq \|\bp\|_{op} \leq 1$ for any $\bp = \left(\begin{smallmatrix} \ba & 0 \\ \bb & \bc \end{smallmatrix} \right) \in P(\cOnu)$ since the following equality holds:
\eq{
\left(\begin{matrix} \pi_\nu^{-nt} \mathbf{I}_m & \\ & \pi_\nu^{mt} \mathbf{I}_n \end{matrix} \right)
\left(\begin{matrix} \ba & 0 \\ \bb & \bc \end{matrix} \right)
\left(\begin{matrix} \pi_\nu^{-nt} \mathbf{I}_m & \\ & \pi_\nu^{mt} \mathbf{I}_n \end{matrix} \right)^{-1}
= \left(\begin{matrix} \ba & 0 \\ \pi_\nu^{(m+n)t} \bb & \bc \end{matrix} \right).
}

Hence, it follows from \eqref{Eq_Msr_Decomp} and Fubini's theorem that for any $t \geq 1$ and decomposable $\bv \in \bigwedge^{i} \mathbf K_\nu^{m+n}$, we have
\eqlabel{Eq_MaxCpt_Horosph1}{
\begin{split}
	&\int_{\SL_d(\cOnu)}\|g_{t} \bk \bv\|^{-\beta_i}d\lam(\bk)	\geq \int_{P(\cOnu)U(\cOnu)}\|g_{t}\bk \bv\|^{-\beta_i}d\lam(\bk) \\
	&= C_\lam \int_{P(\cOnu) \times U(\cOnu)}\|g_{t} \bp \bu \bv\|^{-\beta_i}d(\lam_P \times \lam_U) (\bp, \bu)  \\
	&= C_\lam \int_{U(\cOnu)} \int_{P(\cOnu)}\|(g_{t} \bp g_{-t})g_{t} \bu \bv\|^{-\beta_i}d\lam_P (\bp)d\lam_U(\bu)  \\
	&\geq C_\lam \int_{U(\cOnu)} \int_{P(\cOnu)}\|g_{t} \bp g_{-t}\|_{op}^{-\beta_i} \|g_{t} \bu \bv\|^{-\beta_i}d\lam_P (\bp)d\lam_U(\bu)  \\
	&= C_\lam \int_{P(\cOnu)} \|g_{t} \bp g_{-t}\|_{op}^{-\beta_i}d\lam_P (\bp) \int_{U(\cOnu)} \|g_{t}\bu \bv\|^{-\beta_i}d\lam_U(\bu). \\
  &\geq C_\lam \int_{U(\cOnu)} \|g_{t}\bu \bv\|^{-\beta_i}d\lam_U(\bu).
\end{split}
}
Using the identification between $M_{m,n}(\cOnu)$ and $U(\cOnu)$, $\lambda_U$ equals $\overline{\mu_{mn}}$. It follows from \eqref{Eq_MaxCpt_Horosph1} and Proposition \ref{OrtEst} that
\eqlabel{Eq_MaxCpt_Horosph3}{
\begin{split}
	\int_{M_{m,n}(\cOnu)} \|g_{t} u_\bs \bv\|^{-\beta_i}d\mu_{mn}(\bs)
	&\leq \frac{1}{C_\lam} \int_{\SL_d(\cOnu)}\|g_{t}\bk\bv\|^{-\beta_i}d\lam(\bk) \\
    &\leq \frac{C}{C_\lam} t\qnu^{-mnt}\|\bv\|^{-\beta_i}.
\end{split}
}

Now, consider the case when $i \in \llbracket m+1, d-1\rrbracket$. 
Note that we can assume that $n \geq 2$ since we can apply \eqref{Eq_MaxCpt_Horosph3} to all $1 \leq i \leq m = d-1$ if $n=1$. Denote
\eq{
	\wt{g}_t = \left(\begin{matrix} \pinu^{-mt} \mathbf{I}_n & 0 \\ 0 & \pinu^{nt}\mathbf{I}_m \end{matrix} \right); \wt{u}_\bs = \left(\begin{matrix} \mathbf{I}_n & \bs \\ 0 & \mathbf{I}_m \end{matrix} \right); l_\bs = \left(\begin{matrix} \mathbf{I}_m & 0 \\ \bs & \mathbf{I}_n \end{matrix} \right); E = \left(\begin{matrix} 0 & \mathbf{I}_n \\ \mathbf{I}_m & 0 \end{matrix} \right)
} for $\bs \in M_{n,m}(\cOnu)$.
Observe that for $\bs \in M_{n,m}(\cOnu)$ and  $\bw \in \bigwedge 
 \mathbf K_{\nu}^{d}$, we have
\eqlabel{Eq_Codim}{
E g_{-t} l_{\bs} \bw = \wt{g}_t \wt{u}_{\bs} E \bw.
}
Thus, for each $i \in \llbracket m+1, d-1 \rrbracket$, we have
 \begin{alignat*}{2} 
	& \ \ \ \ \int_{M_{m,n}(\cOnu)} \|g_{t} u_\bs \bv\|^{-\beta_i}d\mu_{mn}(\bs)   \quad  \quad && \\
	& =\int_{M_{m,n}(\cOnu)} \|*(g_{t} u_\bs \bv)\|^{-\beta_i}d\mu_{mn}(\bs)   \quad \quad \quad \quad \quad \quad && \text{by Lemma \ref{lem_dual} \eqref{lem_dual_1} }  \\
	& =\int_{M_{m,n}(\cOnu)} \|g_{-t} l_{(-{}^{t}\bs\!)} (*\bv)\|^{-\beta_i}d\mu_{mn}(\bs)   \quad \quad && \text{by Lemma \ref{lem_dual} \eqref{lem_dual_2} }   \\
	& =\int_{M_{m,n}(\cOnu)} \|Eg_{-t} l_{(-{}^{t}\bs\!)} (*\bv)\|^{-\beta_i}d\mu_{mn}(\bs)   \quad  \quad && \text{by \cite[Lemma 2.5]{Bang}} \\
	& =\int_{M_{m,n}(\cOnu)} \|\wt{g}_t  \wt{u}_{(-{}^{t}\bs\!)} E(*\bv)\|^{-\beta_i}d\mu_{mn}(\bs)   \quad \quad && \text{by (\ref{Eq_Codim})} \\
	& =\int_{M_{n,m}(\cOnu)} \|\wt{g}_t  \wt{u}_\bs E(*\bv)\|^{-\beta_i}d\mu_{nm}(\bs)   \quad \quad &&  \\
	& \leq \wt{C} t\qnu^{-mnt}\|E(*\bv) \|^{-\beta_i} = \wt{C} t\qnu^{-mnt}\| \bv \|^{-\beta_i}, \quad  \quad && 
 \end{alignat*} where the last inequality follows from \eqref{Eq_MaxCpt_Horosph3} by exchanging the roles of $m$ and $n$, by replacing $i$ by $d-i \in \llbracket 1, n\rrbracket$, $\bv$ by $E(*\bv) \in \bigwedge^{d-i} \mathbf K_\nu^{d}$, and $\beta_i$ by $\frac{n}{d-i}$.
\end{proof}

\section{Construction of Margulis functions}\label{sec5}
In this section, we construct Margulis functions satisfying a contraction hypothesis with an optimal contraction rate, which 
implies the upper bound of the Hausdorff dimension of the set of singular systems of linear forms over global function fields. 
We follow the construction in \cite[Section 3]{KKLM}.

Recall that $X=G/\Ga$ is the $G$-orbit of $R_\nu^d$ in the space of unimodular lattices in $\mb{K}_\nu^d$ as in Subsection \ref{subsec2.2}. For a lattice $x\in X$, we say that a subspace $L$ of $\mb{K}_\nu^d$ is \textit{$x$-rational} if $L\cap x$ is a lattice in $L$.
For any $i \in \llbracket 1, d \rrbracket$ and $x \in X$, let us denote by $F_i(x)$ the set of all $i$-dimensional $x$-rational subspaces in $\mb{K}_{\nu}^{d}$.
For any $L \in F_i(x)$, let $\|L\| = \| \bv_1 \wedge \cdots \wedge \bv_i\|$ for an $R_\nu$-basis $\{\bv_1,\dots,\bv_i\}$ of $L\cap x$, which is independent of the choice of basis. 
Define 
\eq{
\alpha_i (x) = \max\left\{\frac{1}{\|L\|} : L\in F_i(x)\right\}.
}
Note that $\al_d(x) = 1$, thus we set $\al_0(x)=1$. 
We remark that the alpha function was constructed in \cite{KKLM} using the covolume of sublattices, which coincides with the norm of the wedge product of a basis in real cases. While the covolume of $L \cap x$ in $L$ is different from $\|L\|$ in global function fields (see, e.g., \eqref{Eq_cov_R}), the norm $\|L\|$ can be seen as the normalized covolume (see \cite{HP} or \cite{Bang}). 

The following corollary is a function field version of \cite[Corollary 3.6]{KKLM}. We assume that $R_\nu$ is a principal ideal domain only since we need to use the submodularity of the covolume function established in \cite{Bang} under this assumption. Therefore, in this section and the next section,\\
\textbf{Assumption:}
\[
R_\nu \text{ is a principal ideal domain}.
\]

\begin{cor}\label{CorAlEst}
Let $i,\beta_i$ and $c>0$ be as in Proposition \ref{LinAlgEst}. 
For any $t \in \bNN$, $x\in X$ and $i\in \llbracket 1, d-1 \rrbracket$, we have
\[\begin{split}
\int_{M_{m,n}(\cO_\nu)} \al_i (g_t u_\bs x)^{\beta_i} d\mu_{mn}(\bs) &\leq c t q_\nu^{-mnt}\al_i(x)^{\beta_i}\\
&+ q_\nu^{2mnt\beta_i} \max_{1\leq j \leq \min\{d-i,i\}} \left(\sqrt{\al_{i+j}(x)\al_{i-j}(x)}\right)^{\beta_i}.
\end{split}\]
\end{cor}
\begin{proof}
Fix $i \in \llbracket 1, d-1 \rrbracket$ and $x\in X$.
Let $L_i \in F_i(x)$ be such that $\al_i(x)=\frac{1}{\|L_i\|}$ and let 
\[
\Psi_i = \{L\in F_i(x) : \|L\| < q_\nu^{2mnt} \|L_i\|\}.
\]
For any $L \in F_i(x) \smallsetminus \Psi_i$ and $\bs \in M_{m,n}(\cO_\nu)$, since $u_\bs \in \mathrm{SL}_d(\cO_\nu)$, it follows for instance from \cite[Lemma 2.5]{Bang} that
\eqlabel{Eq_nbd1}{
\|u_\bs L_i\| = \|L_i\| \leq q_\nu^{-2mnt}\|L\| = q_\nu^{-2mnt}\|u_\bs L\|.
}
Observe that for any $j\in \llbracket 1, d \rrbracket$ and $I \in \wp^{d}_{j}$, we have $q_\nu^{-mnt} \leq \|g_t \be_I\| \leq q_\nu^{mnt}$, hence for any $\bv \in\bigwedge^j \mathbf K_\nu^d$, it follows that
\eqlabel{Eq_nbd2}{
q_\nu^{-mnt}\|\bv\| \leq \|g_t \bv\| \leq q_\nu^{mnt}\|\bv\|
}
Thus it follows from \eqref{Eq_nbd1} and \eqref{Eq_nbd2} that
\eqlabel{Eq_nbd3}{
\|g_t u_\bs L_i\| \leq q_\nu^{mnt}\|u_\bs L_i\| \leq q_\nu^{-mnt}\|u_\bs L\| \leq \|g_t u_\bs L\|.
}
We consider two cases $\Psi_i = \{L_i\}$ and $\Psi_i \supsetneq \{L_i\}$, separately.\\ \\ 
\textbf{Case $\Psi_i = \{L_i\}$}: Using \eqref{Eq_nbd3} and Proposition \ref{LinAlgEst}, we have
\eq{
\begin{split}
\int_{M_{m,n}(\cO_\nu)} \al_i(g_t u_\bs x)^{\beta_i} d\mu_{mn}(\bs) 
&\leq \int_{M_{m,n}(\cO_\nu)} \|g_t u_\bs L_i \|^{-\beta_i} d\mu_{mn}(\bs) \\
&\leq c t q_\nu^{-mnt} \|L_i\|^{-\beta_i} =ctq_\nu^{-mnt}\al_i(x)^{\beta_i}.
\end{split}
}
\textbf{Case $\Psi_i \supsetneq \{L_i\}$}: Take $L' \in \Psi_i$ such that $L' \neq L_i$. Since $\dim(L'+L_i) = i + j$ for some $j >0$, it follows from \cite[Theorem 1.4]{Bang} using the above assumption that for all $\bs \in M_{m,n}(\cO_\nu)$,
\eq{
\begin{split}
\al_i(g_t u_\bs x) &\leq q_\nu^{mnt}\al_i(x) = \frac{q_\nu^{mnt}}{\|L_i\|} \leq \frac{q_\nu^{2mnt}}{\sqrt{\|L'\|\|L_i\|}}\\
&\leq \frac{q_\nu^{2mnt}}{\sqrt{\|L' \cap L_i\| \|L' +L_i\|}} \leq q_\nu^{2mnt}\sqrt{\al_{i+j}(x)\al_{i-j}(x)}.
\end{split}
}
Thus it follows from $\mu_{mn}(M_{m,n}(\cO_\nu))=1$ that
\eq{
\int_{M_{m,n}(\cO_\nu)} \al_i (g_t u_\bs x)^{\beta_i} d\mu_{mn}(\bs) \leq q_\nu^{2mnt\beta_i}\max_{1\leq j \leq \min\{d-i,i\}} \left(\sqrt{\al_{i+j}(x)\al_{i-j}(x)}\right)^{\beta_i}.
}
Combining the two cases, we obtain Corollary \ref{CorAlEst}. 
\end{proof}

Using Corollary \ref{CorAlEst} and \cite[Proposition 4.1]{KKLM}, we have the following corollary.
\begin{cor}\label{CorMarFtn}
Let $\beta_i$'s and $c>0$ be as in Proposition \ref{LinAlgEst}. For any $t \in \bNN$, one can choose constants $\om_0,\dots,\om_d >0 $ and $T>0$ such that for any $x\in X$ with $\wt{\al} = \sum_{i=0}^{d} (\om_i \al_i)^{\beta_i}$ satisfying $\wt{\al}(x) > T$, we have 
\[
\int_{M_{m,n}(\cO_\nu)} \wt{\al}(g_t u_\bs x) d\mu_{mn}(\bs) \leq 3ctq_\nu^{-mnt}\wt{\al}(x).
\]
\end{cor}
\begin{proof} Fix $t\in \bNN$.
To use \cite[Proposition 4.1]{KKLM}, we set $H=X$, $d=m+n$, $f_i = \al_i$ so that $f_0=f_d=1$, the linear operator $A$ on the linear space of real functions on $H$ given by 
\[
(Af)(x) = \int_{M_{m,n}(\cO_\nu)} f(g_t u_\bs x) d\mu_{mn}(\bs) 
\] for $\beta(i)=1/\beta_i$. Note that $i\mapsto \beta(i)$ is concave. It follows from Corollary \ref{CorAlEst} that for each $i \in \llbracket 1, d-1 \rrbracket$, we have
\[
\begin{split}
A(f_i^{\beta_i})(x) &= \int_{M_{m,n}(\cO_\nu)}\al_i(g_t u_\bs x)^{\beta_i} d\mu_{mn}(\bs)  \\
&\leq ctq_\nu^{-mnt} f_i^{\beta_i}(x) + \left(q_\nu^{mnt}\right)^{2\beta_i} \max_{1\leq j \leq \min\{d-i,i\}} \left(\sqrt{f_{i+j}(x)f_{i-j}(x)}\right)^{\beta_i}.
\end{split}
\]
Setting $a=ctq_\nu^{-mnt}$, $\om=q_\nu^{mnt}$, and $a'=2ctq_\nu^{-mnt}$, and using \cite[Proposition 4.1]{KKLM}, we can choose 
$\om_0,\dots,\om_d >0 $ and $C_0$ such that $f = \sum_{i=0}^{d} (\om_i f_i)^{\beta_i}$ satisfies that
\[
(Af)(x) \leq a' f(x) + C_0 \quad \text{for all } x\in X.
\]
Thus the function $\wt{\al} = \sum_{i=0}^{d} (\om_i \al_i)^{\beta_i}$ satisfies that for all $x\in X$
\[
\int_{M_{m,n}(\cO_\nu)} \wt{\al}(g_t u_\bs x) d\mu_{mn}(\bs) \leq 2ctq_\nu^{-mnt} \wt{\al}(x) + C_0.
\]
By choosing $T=\frac{C_0 q_\nu^{mnt}}{ct}$, we can conclude Corollary \ref{CorMarFtn}.
\end{proof}

\section{Covering estimates using Margulis function}\label{sec6}
In this section, we fix $c>0$ as in Proposition \ref{LinAlgEst} and for a given $t \in \bNN$, fix $\om_1,\dots,\om_d$ and $T$ as in Corollary \ref{CorMarFtn}, and write $\wt{\al} = \sum_{i=0}^{d} (\om_i \al_i)^{\beta_i}$.  
For any $x\in X$, $M>0$, $N\in \bNN$, let us define
\[
Z_x (M,N,t) := \{\bs \in M_{m,n}(\cO_\nu): \forall \ell \in \llbracket 1, N\rrbracket, ~ \wt{\al}(g_{\ell t} u_\bs x) >M\}.
\]
By Mahler's compactness criterion for the space of unimodular lattices over global function fields (e.g., see \cite[Theorem 1.1]{KST}),
$\wt{\al}$ is a proper map, hence the set 
\[
X_{\leq M} := \{x\in X : \wt{\al} (x) \leq M\} 
\]
is compact. Similarly, we define $X_{>M} = X \smallsetminus X_{\leq M}$.

First observe that for any $\bs\in M_{m,n}(\cO_\nu)$ and $x\in X$, it follows from the fact that $u_\bs \in \SL_d(\cOnu)$ and from \cite[Lemma 2.5]{Bang} that
\eqlabel{Eq_Norm_Inv}{
\wt{\al}(u_\bs x)=\wt{\al}(x).
}
We have the following main proposition. Unlike in the previous sections, from this point on, $\bs_i$ are matrices rather than column vectors of a matrix. We also denote $d\mu_{mn}(\bs)$ by $d\bs$ for simplicity.
\begin{prop}\label{EstProp}
	Let $t \in \bNN$ be fixed. For all $x\in X_{>T}$, $N\in \bNN$, and $M\geq T$, we have
\[
\int_{Z_x (M,N-1,t)} \wt{\al}(g_{Nt}u_\bs x) d\mu_{mn}(\bs) \leq \left(3ct q_\nu^{-mnt}\right)^N \wt{\al}(x).
\]
\end{prop}
\begin{rem}
This proposition is a function field version of \cite[Proposition 5.1]{KKLM}. It is worth noting that while a certain recursive property of Gaussian measures was used in \cite{KKLM}, we use convolutions of measures inspired by \cite{GS}.
\end{rem}
\begin{proof}
For a fixed $t \in \bNN$, let us denote 
\[
Z=\{ (\bs_1,\dots, \bs_N)\in M_{m,n}(\cO_\nu)^N : \wt{\al}(g_t u_{\bs_k}\cdots g_t u_{\bs_1} x) > T \quad \forall k \in \llbracket 1, N-1 \rrbracket \}.
\]
Using Corollary \ref{CorMarFtn} repeatedly, we have
\eqlabel{EqMultInt}{
\int\cdots\int_{Z} \wt{\al}(g_t u_{\bs_N}\cdots g_t u_{\bs_1}x) d\bs_N \cdots d\bs_1
\leq \left(3ctq_\nu^{-mnt}\right)^N \wt{\al}(x).
}
Observe that 
\eq{
g_t u_{\bs_N}\cdots g_t u_{\bs_1}= g_{Nt} u_{\phi(\bs_1,\dots,\bs_N)},
}
where $\phi(\bs_1,\dots,\bs_N)=\sum_{i=1}^N \pi_\nu^{(i-1)dt}\bs_i$.
Thus we have 
\eqlabel{EqIntegra}{
\begin{split}
&\int\cdots\int_{Z} \wt{\al}(g_t u_{\bs_N}\cdots g_t u_{\bs_1}x) d\bs_N \cdots d\bs_1
=\int\cdots\int_{Z} \wt{\al}(g_{Nt} u_{\phi(\bs_1,\dots,\bs_N)}x) d\bs_N \cdots d\bs_1
\end{split}
}
Following \cite[Section 5]{GS}, for each $i \in \llbracket 1, N \rrbracket$, let $\wt{\mu}_i$ be the Radon measure on $M_{m,n}(\cO_\nu)$ defined by 
\[
\int_{M_{m,n}(\cO_\nu)}f(\bs)d\wt{\mu}_i(\bs) = \int_{M_{m,n}(\cO_\nu)} f(\pi_\nu^{(i-1)dt}\bs)d\bs 
\]
for $f \in C(M_{m,n}(\cO_\nu)) $. Note for future use that $\wt{\mu}_1 = \overline{\mu_{mn}}$. Let $\wt{\mu}^{\ast N} = \wt{\mu}_N \ast \cdots\ast \wt{\mu}_1$ be the measure on
$M_{m,n}(\cO_\nu)$ defined by iterated convolutions.
Observe that 
\eqlabel{Eq_CovInt}{\wt{\mu}^{\ast N}(f) = \int_{M_{m,n}(\cO_\nu)}\cdots\int_{M_{m,n}(\cO_\nu)} f(\phi(\bs_1,\dots,\bs_N))d\bs_1\cdots d\bs_N.}
\vspace{0.3cm}\\ 
\textbf{Claim 1.}\; $\{(\bs_1,\dots,\bs_N)\in M_{m,n}(\cO_\nu)^N : \phi(\bs_1,\dots,\bs_N) \in \phi(Z)\} =Z$.
\begin{proof}[Proof of Claim 1]
Clearly, the set $Z$ is contained in the left hand side set. If $(\bs_1,\dots,\bs_N)$ is an element of the left hand side set, 
then there exists $(\bs_1',\dots,\bs_N')\in Z$ such that $\phi(\bs_1,\dots,\bs_N)=\phi(\bs_1',\dots,\bs_N')$.
Hence for any $k \in \llbracket 1, N-1 \rrbracket$,
\eq{
\begin{split}
\wt{\al}(g_t u_{\bs_k}\cdots g_t u_{\bs_1}x) 
&=\wt{\al}(g_t u_{\bs_k}\cdots g_t u_{\bs_1} (g_t u_{\bs_k'}\cdots g_t u_{\bs_1'})^{-1} g_t u_{\bs_k'}\cdots g_t u_{\bs_1'} x) \\
&=\wt{\al}((g_{kt} u_{\sum_{i=1}^k \pi_\nu^{(i-1)dt}(\bs_i-\bs_i')} g_{-kt}) g_t u_{\bs_k'}\cdots g_t u_{\bs_1'} x) \\
&=\wt{\al}(u_{\pi_\nu^{-kdt}\sum_{i=1}^k \pi_\nu^{(i-1)dt}(\bs_i-\bs_i')} g_t u_{\bs_k'}\cdots g_t u_{\bs_1'} x) \\
&=\wt{\al}(u_{\pi_\nu^{-kdt}\sum_{i=k+1}^N \pi_\nu^{(i-1)dt}(\bs_i'-\bs_i)} g_t u_{\bs_k'}\cdots g_t u_{\bs_1'} x)\\
&=\wt{\al}(g_t u_{\bs_k'}\cdots g_t u_{\bs_1'} x) >T.
\end{split}
}
The fourth line holds since $\phi(\bs_1,\dots,\bs_N)=\phi(\bs_1',\dots,\bs_N')$ and the last line since $\pi_\nu^{-kdt}\sum_{i=k+1}^N \pi_\nu^{(i-1)dt}(\bs_i'-\bs_i) \in M_{m,n}(\cO_\nu)$ and \eqref{Eq_Norm_Inv}. Thus $(\bs_1,\dots,\bs_N)\in Z$.
\end{proof}
Let us denote by $\vphi_x(\bs) = \mathds{1}_{\phi(Z)}(\bs) \wt{\al}(g_{Nt}u_\bs x)$. Using Fubini's theorem, it follows from \eqref{EqIntegra}, \eqref{Eq_CovInt}, and \textbf{Claim 1} that
\eqlabel{Eq_Int}{
\int\cdots\int_{Z} \wt{\al}(g_t u_{\bs_N}\cdots g_t u_{\bs_1}x) d\bs_N \cdots d\bs_1= \wt{\mu}^{\ast N}(\vphi_x).
} 
While $\wt{\mu}^{\ast N}$ is absolutely continuous with respect to $\overline{\mu_{mn}}$, the ultrametric property implies that they coincide. Such an exact coincidence does not occur in the real case. See \cite[Lemma 5.5]{GS}.
\vspace{0.3cm}\\ 
\textbf{Claim 2.}\; 
We have $\wt{\mu}^{\ast N}=\overline{\mu_{mn}}$.
\begin{proof}[Proof of Claim 2]

Since $\wt{\mu}_1=\overline{\mu_{mn}}$ is invariant under translations on $M_{m,n}(\cO_\nu)$,
it follows that for any $f\in C(M_{m,n}(\cO_\nu))$, 
\[\begin{split}
\wt{\mu}^{\ast N}(f) &= \int_{M_{m,n}(\cO_\nu)}\cdots\int_{M_{m,n}(\cO_\nu)} f(\bs_1 +\sum_{i=2}^N \pi_\nu^{(i-1)dt}\bs_i)d\bs_1\cdots d\bs_N\\
&=\int_{M_{m,n}(\cO_\nu)}\cdots\int_{M_{m,n}(\cO_\nu)} f(\bs_1 )d\bs_1\cdots d\bs_N\\
&=\overline{\mu_{mn}}(f).
\end{split}\]

\end{proof}
It follows from \textbf{Claim 2} and the fact that $\phi(Z) \subset M_{m,n}(\cO_\nu)$ that
\eqlabel{Eq_Int2}{
	\wt{\mu}^{\ast N} (\vphi_x) =\overline{\mu_{mn}} (\vphi_x) = \int \mathds{1}_{\phi(Z)}(\bs) \wt{\al}(g_{Nt}u_\bs x) d\bs.
}
We need the following claim.\vspace{0.3cm}\\ 
\textbf{Claim 3.}\; For any $\bs\in Z_x (M,N-1,t)$, $\phi^{-1}(\bs) \subset Z$.
\begin{proof}[Proof of Claim 3]
If $(\bs_1,\dots,\bs_N) \in \phi^{-1}(\bs)$, then $\sum_{i=1}^N \pi_\nu^{(i-1)dt}\bs_i \in Z_x(M,N-1,t)$, 
hence for all $\ell \in \llbracket 1, N-1 \rrbracket$,
\eqlabel{Eq_Cl1_1}{
\wt{\al}\left(\left(g_{\ell t}u_{\sum_{i=\ell+1}^N \pi_\nu^{(i-1)dt}\bs_i}g_{-\ell t}\right)g_{\ell t} u_{\sum_{i=1}^\ell \pi_\nu^{(i-1)dt}\bs_i} x\right)>M.
}
Note that 
\eqlabel{Eq_Cl1_2}{
g_{\ell t}u_{\sum_{i=\ell+1}^N \pi_\nu^{(i-1)dt}\bs_i}g_{-\ell t} = u_{\sum_{i=\ell+1}^N \pi_\nu^{(i-1-\ell)dt}\bs_i}
}
and $\sum_{i=\ell+1}^N \pi_\nu^{(i-1-\ell)dt}\bs_i \in M_{m,n}(\cO_\nu)$. Thus it follows from \eqref{Eq_Norm_Inv}, \eqref{Eq_Cl1_1}, and \eqref{Eq_Cl1_2} that
\eq{
\wt{\al}(g_{\ell t} u_{\sum_{i=1}^\ell \pi_\nu^{(i-1)dt}\bs_i} x) =\wt{\al}(g_t u_{\bs_\ell}\cdots g_t u_{\bs_1}x)> M\geq T,
}
which proves $(\bs_1,\dots,\bs_N)\in Z$.
\end{proof}
Combining \eqref{EqMultInt}, \eqref{Eq_Int}, \eqref{Eq_Int2}, and \textbf{Claim 3}, we have
\[
\begin{split}
\left(3ctq_\nu^{-mnt}\right)^N \wt{\al}(x) \geq \wt{\mu}^{\ast N} (\vphi_x) 
&\geq \int \mathds{1}_{\phi(Z)}(\bs) \wt{\al}(g_{Nt}u_\bs x) d\bs\\ 
&\geq \int_{Z_x(M,N-1,t)}\wt{\al}(g_{Nt}u_\bs x) d\bs.
\end{split}
\]
\end{proof}

\begin{cor}\label{CorEst}
Let $t \in \bNN$ be fixed. There exists $M_0 =M_0(t)>T$ such that for all $x\in X$, $N\in \bNN$, $M>M_0$, the set $Z_x(M,N,t)$ can be covered by 
$\frac{\wt{\al}(x)}{M}(3ct)^N q_\nu^{(m+n-1)mntN}$ balls in $M_{m,n}(\cO_\nu)$ of radius $q_\nu^{-(m+n)tN}$.
\end{cor}
\begin{proof}
Using \eqref{Eq_nbd2}, \eqref{Eq_Norm_Inv}, and the construction of $\wt{\al}$ in Corollary \ref{CorMarFtn}, 
it follows that for any $x\in X$ and $\bs \in M_{m,n}(\cO_\nu)$, we have
\eqlabel{eqalphabound}{
\wt{\al}(g_t u_\bs x) \leq  \max_i q_\nu^{mnt\beta_i} \wt{\al}(x).
}
Setting $M_0=M_0(t)= T  \max_i q_\nu^{mnt\beta_i}>T$, it follows that $\wt{\al}(g_t u_\bs x) \leq M_0$ for any $x\in X_{\leq T}$ and $\bs \in M_{m,n}(\cO_\nu)$. Hence, if $x\in X_{\leq T}$ and $M>M_0$, then the set $Z_x(M,N,t)$ is empty, so the corollary follows. 

For fixed $x\in X_{>T}$ and $M>M_0>T$, it follows from Proposition \ref{EstProp} that 
\[
\int_{Z_x(M,N-1,t)}\wt{\al}(g_{Nt}u_\bs x)d\bs \leq (3ctq_\nu^{-mnt})^{N}\wt{\al}(x),
\]
hence 
\eqlabel{Eq_Intupper}{
\mu_{mn}(Z_{x}(M,N,t)) \leq \frac{\wt{\al}(x)}{M}(3ct)^N q_\nu^{-mntN}.
}

Now observe that for any $\ell \in \bNN$ we can partition $\cO_\nu$ into $q_\nu^\ell$ disjoint balls of radius $q_\nu^{-\ell}$ as follows:
\eqlabel{eqPartition}{
\cO_\nu = \bigsqcup_{a_0,\dots, a_{\ell-1} \in \mathbf k_\nu} \sum_{i=0}^{\ell-1}a_i \pi_\nu^{i} + \pi_\nu^{\ell} \cO_\nu.
}
Thus we can partition $M_{m,n}(\cO_\nu)$ into $p:=q_\nu^{(m+n)mntN}$ disjoint balls $D_1,\dots,D_p$ of radius 
$q_\nu^{-(m+n)tN}$. Let $p_1$ be the number of the balls $D_i$'s contained in $Z_{x}(M,N,t)$. 
Since $\mu_{mn}(D_i)=q_\nu^{-mn(m+n)tN}$ for all $i$,  it follows from \eqref{Eq_Intupper} that 
\eq{
p_1 \leq \frac{\wt{\al}(x)}{M}(3ct)^N q_\nu^{mn(m+n-1)tN}.
}
Reordering the $D_i$'s if necessary, we may assume that $\{D_1,\dots,D_{p_1}\} = \{D_i : D_i \subset Z_{x}(M,N,t)\}$.

Now fix any $i>p_1$, then we can take $\bs_0 \in D_i$ such that $\bs_0 \notin Z_{x}(M,N,t)$, 
which means that $\wt{\al}(g_{\ell t}u_{\bs_0} x)\leq M$ for some $\ell \in \llbracket 1, N\rrbracket$. 
By the ultrametric property, 
for any $\bs\in D_i$, $\bs=\bs_1 + \bs_0$ for some $\bs_1$ with $\|\bs_1\|\leq q_\nu^{-(m+n)tN}$. 
Hence 
$$g_{\ell t} u_\bs = g_{\ell t}u_{\bs_1}g_{-\ell t} g_{\ell t} u_{\bs_0} = u_{\pi_\nu^{-(m+n)\ell t}\bs_1} g_{\ell t} u_{\bs_0},$$
which implies $\wt{\al}(g_{\ell t}u_\bs x) = \wt{\al}(g_{\ell t}u_{\bs_0} x) \leq M$ by \eqref{Eq_Norm_Inv}.
Thus
\[
Z_x(M,N,t) \subset \bigcup_{i=1}^{p_1}D_i,
\]
which proves the corollary.
\end{proof}

For a given compact subset $Q$ of $X$, $N\in \bNN$, $\del\in(0,1)$, and $x\in X$, define the set
\[
Z_x (Q,N,t,\del) := \left\{\bs \in M_{m,n}(\cO_\nu): \frac{1}{N}|\{\ell\in \llbracket 1, N\rrbracket:g_{t\ell}u_\bs x\notin Q\}|\geq \del\right\}.
\]
\begin{thm}\label{MainCounting}
There exist $t_0>0$ and a function $ C'':X\to \bR_+$ such that the following holds. For any $t>t_0$, there exists a compact set $Q=Q(t)$ of $X$ such that for any $N\in\bNN$, $\delta\in(0,1)$, and $x\in X$, the set $Z_x(Q,N,t,\del)$ can be covered by 
\[
C'' (x)t^{3N} q_\nu^{(m+n-\del)mntN}
\] balls in $M_{m,n}(\cO_\nu)$ of radius $q_\nu^{-(m+n)tN}$.
\end{thm}
\begin{proof}
Let $t>t_0$ be given, with $t_0 \geq 2 $ a large real number, to be determined later on, depending only on $m,n$. Let $\wt{\al}$ and $T$ be as in Corollary \ref{CorMarFtn}. We will find large enough $M>0$ such that the compact set $Q=X_{\leq M}$ satisfies the conclusion of the theorem.

For a given $N\in\bNN$ and $x\in X$, let
\[
J_x =\{\ell\in\llbracket 1, N\rrbracket:g_{t\ell}x \notin Q\}.
\]
Then, 
\[
Z_x(Q,N,t,\del)=\{\bs\in M_{m,n}(\cO_\nu) : |J_{u_\bs x}|\geq \del N\}.
\]
For any subset $J\subset\llbracket 1, N\rrbracket$, let $Z(J)=\{\bs\in M_{m,n}(\cO_\nu):J_{u_\bs x}=J\}$. Note that $Z_x(Q,N,t,\del)=\bigcup_J Z(J)$, where the union runs over all subsets $J\subset \llbracket 1, N\rrbracket$ with $|J|\geq \del N$. Since the number of such subsets is at most $2^N \leq t^N$, it is enough to show that for a given subset $J\subset\llbracket 1, N\rrbracket$, the set $Z(J)$ can be covered by $ C''(x) t^{2N}q_\nu^{((m+n)N-|J|)mnt}$ balls of radius $q_\nu^{-(m+n)tN}$, for 
\[
C''(x)=\max\left\{1,\max\{\al_i^{\beta_i}(x):i \in \llbracket 1, m+n-1 \rrbracket\}\right\}.
\]

We decompose the sets $J$ into ordered subintervals $J_1,\dots,J_p$ of maximal size such that 
$J=\bigsqcup_{i=1}^p J_i$. Let $I_1,\dots,I_{p'}$ be the ordered maximal subintervals of $\llbracket 1, N\rrbracket \smallsetminus J$
such that
\[
\llbracket 1, N\rrbracket=\bigsqcup_{i=1}^p J_i \sqcup \bigsqcup_{j=1}^{p'} I_j.
\]
\textbf{Claim.}\; For any integer $L\leq N$, if
\eqlabel{eqIndStep}{
\llbracket 1, L\rrbracket =\bigsqcup_{i=1}^\ell J_i \sqcup \bigsqcup_{j=1}^{\ell'} I_j,
}
then the set $Z(J)$ can be covered by
\eqlabel{eqCoverNum}{
C''(x) t^{2L}q_\nu^{((m+n)L-|J_1|-\cdots-|J_\ell|)mnt}
} balls of radius $q_\nu^{-(m+n)tL}$.
\begin{proof}[Proof of Claim]
We prove the claim by induction on $L$. In the first step, if $\llbracket 1, L\rrbracket =J_1$, then by choosing $t_0 \geq 3c$ and $M>M_0$ large enough so that $M\geq \sum_{i=1}^d \om_i^{\beta_i}$ hence $\frac{\wt{\al}(x)}{M}\leq C'' (x)$, 
it follows from Corollary \ref{CorEst} that the set $Z_x(M,L,t)$ can be covered by 
\[
\frac{\wt{\al}(x)}{M}(3ct)^L q_\nu^{(m+n-1)mntL}\leq C''(x) t^{2L} q_\nu^{((m+n)L - |J_1|)mnt}
\]
balls of radius $q_\nu^{-(m+n)tL}$.
Since $Q=X_{\leq M}$ and $\llbracket 1, L\rrbracket\subset J$, we have $$Z(J)=\{\bs\in M_{m,n}(\cO_\nu):\forall \ell\in J,~ \wt{\al}(g_{\ell t}u_{\bs}x)>M\} \subset Z_x(M,L,t).$$ Thus the claim follows.
If $\llbracket 1, L\rrbracket =I_1$, then it follows from \eqref{eqPartition} that $Z(J)\subset M_{m,n}(\cO_\nu)$ can be covered by $q_\nu^{(m+n)Lmnt} \leq C''(x)t^{2L}q_\nu^{(m+n)Lmnt} $ balls of radius $q_\nu^{-(m+n)tL}$.

Now, assume that the claim holds for some $L \leq N $ satisfying \eqref{eqIndStep} and denote the covering for $Z(J)$ by $\{B_i\}_{i=1}^p$ for $p$ given by \eqref{eqCoverNum}. In the induction step, for the next $L'>L$, we have
two cases: either
\eqlabel{eqPartCase1}{
\llbracket 1, L'\rrbracket=\llbracket 1, L\rrbracket \sqcup I_{\ell'+1}
} or
\eqlabel{eqPartCase2}{
\llbracket 1, L'\rrbracket=\llbracket 1, L\rrbracket \sqcup J_{\ell+1}.
} 
In the first case \eqref{eqPartCase1}, applying \eqref{eqPartition}, each ball $B_i$ in $M_{m,n}(\cO_\nu)$ of radius $q_\nu^{-(m+n)tL}$ can be covered by $q_\nu^{(m+n)t|I_{\ell'+1}|mn}$ balls of radius $q_\nu^{-(m+n)t(L+|I_{\ell'+1}|)}$. Since $L'=L+|I_{\ell'+1}|$, $Z(J)$ can be covered with $$ C'' (x)t^{2L'}q_\nu^{((m+n)L'-|J_1|-\cdots-|J_\ell|)mnt}$$ balls of radius $q_\nu^{-(m+n)tL'}$, as wanted.

Now assume the second case \eqref{eqPartCase2}. 
Fix any ball $B_i$ with $Z(J)\cap B_i \neq \varnothing$ and fix $\bs_0\in Z(J)\cap B_i$. Then we have $\wt{\al}(g_{Lt}u_{\bs_0} x) \leq M$ since $L\notin J$ by the maximality of $J_{\ell+1}$ but $\wt{\al}(g_{jt}u_{\bs_0} x) > M$ for all $j\in J_{\ell+1} \subset J $. 
Hence, it follows from \eqref{eqalphabound} that
\[
M\geq \wt{\al}(g_{Lt}u_{\bs_0} x) \geq \min_i q_\nu^{-mnt\beta_i}\wt{\al}(g_{(L+1)t}u_{\bs_0} x) > \min_i q_\nu^{-mnt\beta_i} M.
\]
For a sufficiently large $M$ so that $\min_i q_\nu^{-mnt\beta_i}M > M_0$ and for $t_0\geq 3c$, Corollary \ref{CorEst} with 
$x'=g_{Lt}u_{\bs_0} x$ so that $\frac{\wt{\al}(x')}{M}\leq 1$ implies that $Z_{x'}(M,|J_{\ell+1}|,t)$ can be covered by
\[
(3ct)^{|J_{\ell+1}|}q_\nu^{(m+n-1)mnt|J_{\ell+1}|} \leq t^{2|J_{\ell+1}|}q_\nu^{(m+n-1)mnt|J_{\ell+1}|}
\]
balls of radius $q_\nu^{-(m+n)t|J_{\ell+1}|}$.

Fix any $\bs\in Z(J)\cap B_i$ and note that $\wt{\al}(g_{jt}u_\bs x)>M$ for all $j\in J$ and $\|\bs-\bs_0\|\leq q_{\nu}^{-(m+n)tL}$. 
Write $\bs'=\bs-\bs_0$ and observe that 
\[
\wt{\al}(g_{jt}u_\bs x)=\wt{\al}(g_{(j-L)t}g_{Lt}u_{\bs'}g_{-Lt}g_{Lt}u_\bs x)=\wt{\al}(g_{(j-L)t}u_{\pi_{\nu}^{-(m+n)tL}\bs'}x').
\]
Since $J_{\ell+1}\subset J$, we have $\pi_{\nu}^{-(m+n)tL}\bs' \in Z_{x'}(M,|J_{\ell+1}|,t)$.
Therefore, we have
\[
Z(J)\cap B_i \subset \pi_\nu^{(m+n)tL} Z_{x'}(M,|J_{\ell+1}|,t) + \bs_0.
\]

Note that the set $ \pi_\nu^{(m+n)tL} Z_{x'}(M,|J_{\ell+1}|,t) + \bs_0$ can be covered by
\[
p'= t^{2|J_{\ell+1}|}q_\nu^{(m+n-1)mnt|J_{\ell+1}|}
\]
balls of radius $q_\nu^{-(m+n)t(L+|J_{\ell+1}|)}=q_\nu^{-(m+n)tL'}$. Hence $Z(J)$ can be covered by
\[
pp'=C''(x) t^{2L'}q_\nu^{((m+n)L'-|J_1|-\cdots-|J_{\ell+1}|)mnt}
\] balls of radius $q_\nu^{-(m+n)tL'}$, which concludes the claim.
\end{proof}
The claim with $L=N$ concludes Theorem \ref{MainCounting} since $|J|\geq \del N$.
\end{proof}

Now we are ready to prove Theorem \ref{Thm_sing}. In fact, we are able to prove a stronger proposition. For this, we introduce the following definition:
Given $\del\in[0,1]$, we say that a point $x\in X$ \textit{$\del$-escapes on average} if for any compact subset $Q$ in X, we have
\[
\lim_{N\to\infty}\frac{1}{N}|\left\{\ell\in \llbracket 1, N\rrbracket:g_{\ell} x\notin Q\right\}|\geq \del.
\]

\begin{prop}\label{Prop_dim}
Let $\del\in[0,1]$ be given. For any $x\in X$, we have
\[
\dim_H \{\bs\in M_{m,n}(\mathbf K_\nu): u_\bs x \ \del\text{-escapes on average}\} \leq mn-\frac{\del mn}{m+n}.
\] Consequently,
\[
\dim_H \{x\in X: x\ \del\text{-escapes on average}\} \leq \dim X -\frac{\del mn}{m+n}.
\]
\end{prop}
\begin{proof}
It is enough to show the proposition for $\del\in (0,1)$.
Let $\del\in(0,1)$ and $x\in X$ be given and let $t_0 >0$ be as in Theorem \ref{MainCounting}. For any $t>t_0$, choose the compact set $Q$ as in Theorem \ref{MainCounting}. Denote \[Z_x=\{\bs\in M_{m,n}(\mathbf K_\nu): u_\bs x \ \del\text{-escapes on average}\}.\]
Let $\del'\in(0,\del)$ and denote by $\overline{\dim}_B$ the upper box-counting dimension. Consider the compact set $Q' = \bigcup_{k=1}^{t}a_{-k}Q$ in $X$. Fix any $\bs \in Z_x \cap M_{m,n}(\cO_\nu)$ and observe that for any large enough $N\geq 1$, we have
$$|\left\{\ell\in \llbracket 1, tN\rrbracket:g_{\ell} u_\bs x\notin Q'\right\}|\geq \del' tN.$$
Since $g_\ell u_\bs x \in Q'$ if and only if $g_{k+\ell}u_\bs x \in Q$ for some $1\leq k\leq t$, if $g_{\ell t}u_\bs x \in Q$, then $g_j u_\bs x \in Q'$ for any $(\ell-1)t+1\leq j\leq \ell t$. Thus, we have
\[
t|\left\{\ell\in \llbracket 1, N\rrbracket:g_{\ell t} u_\bs x\in Q\right\}| \leq |\left\{\ell\in \llbracket 1, tN\rrbracket:g_{\ell} u_\bs x\in Q'\right\}|.
\]
This implies that
\[
|\left\{\ell\in \llbracket 1, N\rrbracket:g_{\ell t} u_\bs x\notin Q\right\}| \geq \del' N.
\]
Therefore, we have 
\[Z_x \cap M_{m,n}(\cO_\nu) \subset \bigcup_{N_0 \geq 1}\bigcap_{N\geq N_0} Z_x(Q,N,t,\del').\]

It follows from Theorem \ref{MainCounting} that 
\[
\begin{split}
\dim_H Z_x &\leq \overline{\dim}_B Z_x \leq \overline{\dim}_B\left( \bigcup_{N_0 \geq 1}\bigcap_{N\geq N_0} Z_x(Q,N,t, \del')\right)\\
&\leq \sup_{N_0\geq 1}\limsup_{N\to\infty}\frac{\log_{q_\nu} (C'' (x)t^{3N}q_\nu^{(m+n- \del')mntN})}{-\log_{q_\nu}(q_\nu^{-(m+n)tN})}\\
&=\frac{3\log_{q_\nu} t + (m+n- \del')mnt}{(m+n)t}.
\end{split}
\] 
Letting $t\to\infty$ and $\del'\to\del$, we have
\[
\dim_H Z_x \leq mn-\frac{\del mn}{m+n}.
\]

The ``consequently'' part follows from the same argument as in the proof of \cite[Theorem 1.1]{KKLM} by considering the thickening along the weak stable horospherical subgroup with respect to $g_1$ (we can use the neighborhood $V$ of $\mb{I}_{m+n}$ contained in $P(\cOnu) U(\cOnu)$ defined in \textbf{Step 3} of the proof of Theorem \ref{ThmWedgeUpperBound}).
\end{proof}
\begin{proof}[Proof of Theorem \ref{Thm_sing}]
    It follows from Dani's correspondence (Lemma \ref{Lem_Dani}) that if $\mb{s}\in M_{m,n}(\mb{K}_\nu)$ is singular, then the point $u_{\mb{s}}x_0$ $1$-escapes on average, where $x_0$ denotes the identity coset in $X$. Therefore, by Proposition \ref{Prop_dim}, we have
    \[\dim_H \mathrm{Sing}(m,n) \leq mn - \frac{mn}{m+n}.\]
\end{proof}

\def\cprime{$'$} \def\cprime{$'$} \def\cprime{$'$}
\providecommand{\bysame}{\leavevmode\hbox to3em{\hrulefill}\thinspace}
\providecommand{\MR}{\relax\ifhmode\unskip\space\fi MR }
\providecommand{\MRhref}[2]{%
  \href{http://www.ams.org/mathscinet-getitem?mr=#1}{#2}
}

\end{document}